\documentclass{article}
\usepackage{amssymb,amsmath,amssymb}
\usepackage{amsthm}
\usepackage{amsfonts}
\usepackage{graphics,psfrag}
\usepackage{pgf,tikz}
\usetikzlibrary{arrows}
\usepackage{hyperref}

\DeclareMathOperator{\val}{val}
\DeclareMathOperator{\rep}{rep}

\theoremstyle{plain}
\newtheorem{theorem}{Theorem}
\newtheorem{lemma}[theorem]{Lemma}
\newtheorem{corollary}[theorem]{Corollary}
\newtheorem{prop}[theorem]{Proposition}

\theoremstyle{definition}
\newtheorem{definition}[theorem]{Definition}
\newtheorem{example}[theorem]{Example}
\newtheorem{remark}[theorem]{Remark}

\def\undertilde#1{\mathord{\vtop{\ialign{##\crcr
$\hfil\displaystyle{#1}\hfil$\crcr\noalign{\kern1.5pt\nointerlineskip}
$\hfil\tilde{}\hfil$\crcr\noalign{\kern1.5pt}}}}}

\usepackage{pgf}
\usepackage{tikz}
\usetikzlibrary{arrows,automata, shapes, positioning}
\usetikzlibrary{decorations}
\usetikzlibrary{snakes}
\usepackage{verbatim}

\newcommand{\seqnum}[1]{\href{http://oeis.org/#1}{\underline{#1}}}

\begin{document}
\begin{center}
\vskip 1cm{\LARGE\bf Counting Subwords Occurrences in Base-$b$ Expansions}  \vskip 1cm
\large
Julien Leroy, Michel Rigo and Manon Stipulanti\\
University of Li\`ege \\ 
Department of Mathematics \\ 
All\'ee de la D\'ecouverte 12 (B37) \\
4000 Li\`ege, Belgium \\
\href{mailto:J.Leroy@ulg.ac.be}{\tt J.Leroy@ulg.ac.be} \\
\href{mailto:M.Rigo@ulg.ac.be}{\tt M.Rigo@ulg.ac.be} \\
\href{mailto:M.Stipulanti@ulg.ac.be}{\tt M.Stipulanti@ulg.ac.be} \\
\end{center}

\vskip .2 in

\begin{abstract}
We count the number of distinct (scattered) subwords occurring in the base-$b$ expansion of the non-negative integers. 
More precisely, we consider the sequence $(S_b(n))_{n\ge 0}$ counting the number of positive entries on each row of a generalization of the Pascal triangle to binomial coefficients of base-$b$ expansions. 
By using a convenient tree structure, we provide recurrence relations for $(S_b(n))_{n\ge 0}$ leading to the $b$-regularity of the latter sequence. 
Then we deduce the asymptotics of the summatory function of the sequence $(S_b(n))_{n\ge 0}$. 
\end{abstract}

\section{Introduction}\label{sec:intro}

A {\em finite word} is a finite sequence of letters belonging to a finite set called the {\em alphabet}. 
The \emph{binomial coefficient} $\binom{u}{v}$ of two finite words $u$ and $v$ is the number of times $v$ occurs as a subsequence of $u$ (meaning as a ``scattered'' subword). 
All along the paper, we let $b$ denote an integer greater than $1$. 
We let $\rep_b(n)$ denote the (greedy) base-$b$ expansion of $n\in\mathbb{N}\setminus\{0\}$ starting with a non-zero digit. 
We set $\rep_b(0)$ to be the empty word denoted by $\varepsilon$. We let 
$$L_b=\{1,\ldots,b-1\}\{0,\ldots,b-1\}^*\cup\{\varepsilon\}$$ 
be the set of base-$b$ expansions of the non-negative integers. 
For all $w\in \{0,\ldots,b-1\}^*$, we also define $\val_b(w)$ to be the value of $w$ in base $b$, i.e., if $w=w_n\cdots w_0$ with $w_i\in\{0,\ldots,b-1\}$ for all $i$, then $\val_b(w)=\sum_{i=0}^n w_i b^i$.

Several generalizations and variations of the Pascal triangle exist and lead to interesting combinatorial, geometrical or dynamical properties \cite{BNS, BS, vonH,JRV,LRS1}.
Ordering the words of $L_b$ by increasing genealogical order, we introduced Pascal-like triangles $\mathrm{P}_b$ \cite{LRS1} where the entry $\mathrm{P}_b(m,n)$ is $\binom{\rep_b(m)}{\rep_b(n)}$. 
Clearly $\mathrm{P}_b$ contains $(b-1)$ copies of the usual Pascal triangle when only considering words of the form $a^m$ with $a\in\{1,\ldots,b-1\}$ and $m\ge 0$.  
In Figure~\ref{fig:S3-debut}, we depict the first few elements of $\mathrm{P}_3$ \seqnum{A284441} and its compressed version highlighting the number of positive elements on each line. 
The data provided by this compressed version is summed up in Definition~\ref{def:S_b}. 

\begin{figure}[h!tb]
    \centering
    \includegraphics[scale=0.3]{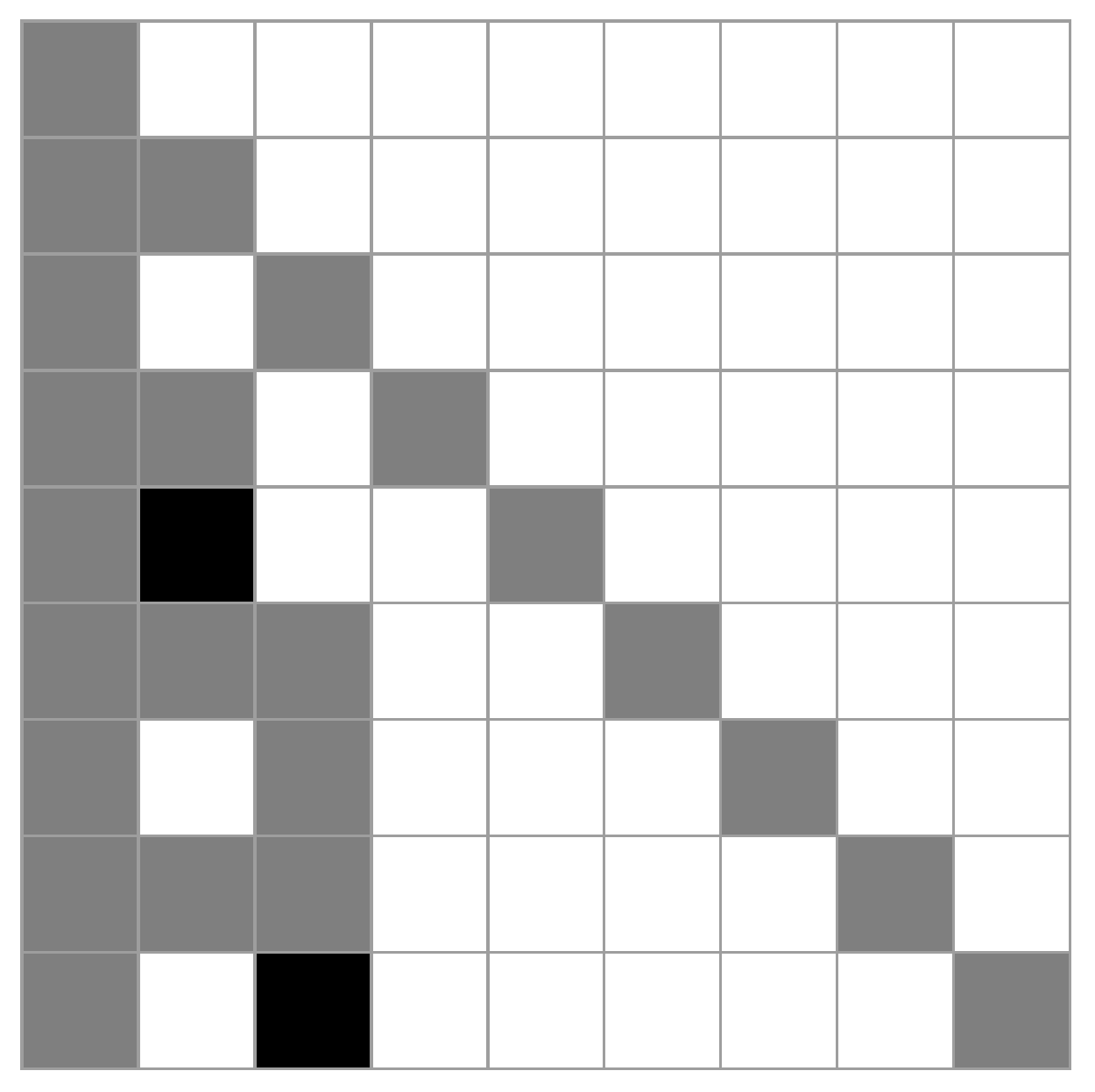} \quad
    \includegraphics[scale=0.3, angle=270, origin=c]{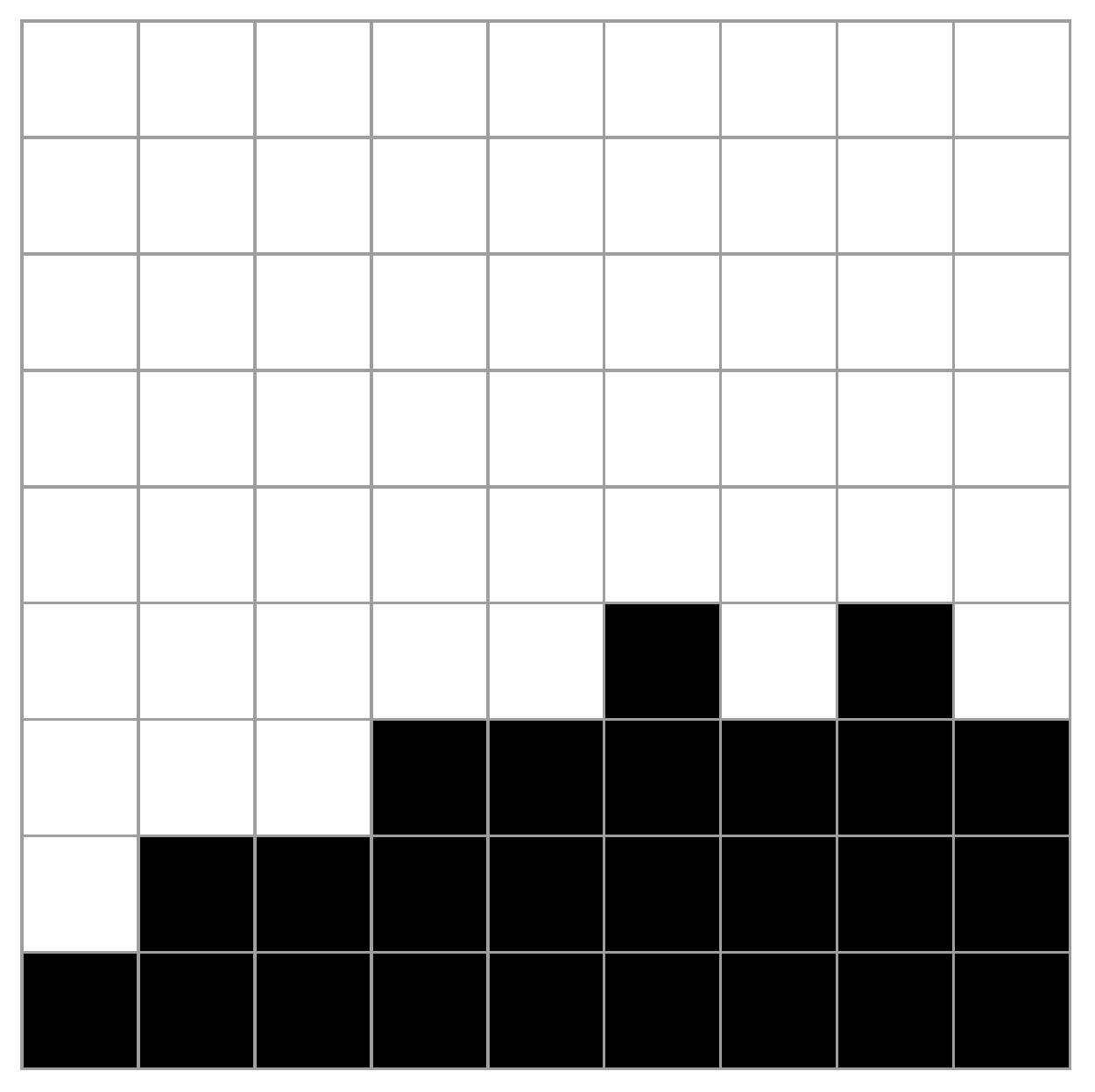}
    \caption{On the left, the first few rows of the generalized Pascal triangle $\mathrm{P}_3$ (a white (resp., gray; resp., black) square corresponds to $0$ (resp., $1$; resp., $2$)) and on the right, its compressed version.}
    \label{fig:S3-debut}
\end{figure}

\begin{definition}\label{def:S_b}
For $n\ge 0$, we define the sequence $(S_b(n))_{n\ge 0}$ by setting
\begin{equation}
    \label{eq:defSb}
    S_b(n):=\#\left\{v\in L_b\mid \binom{\rep_b(n)}{v}>0\right\}.
\end{equation}
We also consider the summatory function $(A_b(n))_{n\ge 0}$ of the sequence $(S_b(n))_{n\ge 0}$ defined by $A_b(0)=0$ and for all $n\ge 1$,
$$A_b(n):=\sum_{j=0}^{n-1} S_b(j).$$
The quantity $A_b(n)$ can be thought of as the total number of base-$b$ expansions occurring as subwords in the base-$b$ expansion of integers less than $n$ (the same subword is counted $k$ times if it occurs in the base-$b$ expansion of $k$ distinct integers).
\end{definition}

In some sense, the sequences $(S_b(n))_{n\ge 0}$ and $(A_b(n))_{n\ge 0}$ measure the sparseness of $\mathrm{P}_b$. 

\begin{example}
If $b=3$, then the first few terms of the sequence $(S_3(n))_{n\ge 0}$ \seqnum{A282715} are
$$1, 2, 2, 3, 3, 4, 3, 4, 3, 4, 5, 6, 5, 4, 6, 7, 7, 6, 4, 6, 5, 7, 
6, 7, 5, 6, 4, 5, 7, 8, 8, 7, 10,\ldots
$$
For instance, the subwords of the word $121$ are $\varepsilon, 1,2,11, 12, 21, 121$. Thus, $S_3(\val_3(121))=S_3(16)=7$.
The first few terms of $(A_3(n))_{n\ge 0}$ \seqnum{A284442} are
$$0,1, 3, 5, 8, 11, 15, 18, 22, 25, 29, 34, 40, 45, 49, 55,\ldots$$ 
\end{example}

We studied \cite{LRS2} the triangle $\mathrm{P}_2$ \seqnum{A282714} and the sequence $(S_2(n))_{n\ge 0}$ \seqnum{A007306}, which turns out to be the subsequence with odd indices of the Stern--Brocot sequence. 
The sequence $(S_2(n))_{n\ge 0}$ is $2$-regular in the sense of Allouche and Shallit \cite{AS99}. 
We studied \cite{LRS3} the behavior of $(A_2(n))_{n \ge 0}$ \seqnum{A282720}.
To this aim, we exploited a particular decomposition of $A_2(2^\ell+r)$, for all $\ell\ge 1$ and all $0\le r< 2^\ell$, using powers of $3$. 

\subsection{Our contribution}

We conjectured six recurrence relations for $(S_3(n))_{n\ge 0}$ depending on the position of $n$ between two consecutive powers of $3$; see \cite{LRS2}. 
Using the heuristic from \cite{AS} suggesting recurrence relations, the sequence $(S_3(n))_{n\ge 0}$ was expected to be $3$-regular.   
It was not obvious that we could derive general recurrence relations for $(S_b(n))_{n\ge 0}$ from the form of those satisfied by $(S_2(n))_{n\ge 0}$.
We thought that $(b-1)b$ recurrence relations should be needed in the general case, leading to a cumbersome statement. 
Moreover it was computationally challenging to obtain many terms of $(S_b(n))_{n\ge 0}$ for large $b$ because the number of words of length $n$ in $L_b$ grows like $b^n$.
Therefore we lack data to conjecture the $b$-regularity of $(S_b(n))_{n\ge 0}$. 

When studying $(A_2(n))_{n\ge 0}$, a possible extension seemed to emerge \cite{LRS3}. 
In particular, we prove that $A_2(2n)=3A_2(n)$ and, sustained by computer experiments, we conjectured that $A_b(nb)=(2b-1)A_b(n)$.

Surprisingly, for all $b\ge 2$, we show in Section~\ref{sec:trie} that the recurrence relations satisfied by $(S_b(n))_{n\ge 0}$ reduce to three forms; see Proposition~\ref{prop:rel-s}.
In particular, this proves the conjecture stated in \cite{LRS2}. Therefore, in Section~\ref{sec:breg}, we deduce the $b$-regularity of $(S_b(n))_{n\ge 0}$; see Theorem~\ref{thm:b-reg}.  Moreover we obtain a linear representation of the sequence with $b\times b$ matrices. 
We also show that $(S_b(n))_{n\ge 0}$ is palindromic over $[(b-1)b^\ell,b^{\ell+1}]$. 

The key to study the asymptotics of $(A_b(n))_{n\ge 0}$ is to obtain specific recurrence relations for this sequence. In Proposition~\ref{prop:rel-A}, we show that theses relations involve powers of $(2b-1)$. Therefore, we prove the conjecture about $A_b(nb)$. In Section~\ref{sec:summatory}, using the so-called $(2b-1)$-decompositions, we may apply the method introduced in \cite{LRS3}. 

We think that this paper motivates the quest for generalized Stern--Brocot sequences and analogues of the Farey tree \cite{Bates, CW, CS, Garrity, Glasby, MGOT}. Namely can one reasonably define a tree structure, or some other combinatorial structure, in which the sequence $(S_b(n))_{n\ge 0}$ naturally appears?

Most of the results are proved by induction and the base case usually takes into account the values of $S_b(n)$ for $0 \leq n < b^2$.
These values are easily obtained from Definition~\ref{def:S_b} and summarized in Table~\ref{tab:initialSb}.

\begin{table}
\centering
$
\begin{array}{|c|c|c|c|c|c|c|c|c|c|c|c|c|c|c|c|}
\hline
\rep_b(n) 	& \varepsilon & x	& x0 & xx & xy & x00 & x0x	& x0y & xx0 & xxx & xxy & xy0 & xyx & xyy & xyz	\\
\hline
S_b(n)		& 1	& 2	& 3	& 3	& 4	& 4 & 5 & 6 & 5 & 4 & 6 & 7 & 7 & 6 & 8	\\
\hline
\end{array}
$
\caption{The first few values of $S_b(n)$ for $0 \leq n < b^3$, with pairwise distinct $x,y,z \in \{1,\dots,b-1\}$.}
\label{tab:initialSb}
\end{table}

\section{General recurrence relations in base $b$}\label{sec:trie}

The aim of this section is to prove the following result exhibiting recurrence relations satisfied by the sequence $(S_b(n))_{n\ge 0}$. This result is useful to prove that the summatory function of the latter sequence also satisfies recurrence relations; see Section~\ref{sec:summatory}.

\begin{prop}\label{prop:rel-s}
The sequence $(S_b(n))_{n\ge 0}$ satisfies $S_b(0)=1$, $S_b(1)=\cdots =S_b(b-1)=2$, and, for all $x,y\in\{1,\ldots,b-1\}$ with $x\ne y$, all $\ell\ge 1$ and all $r\in\{0,\ldots,b^{\ell-1}-1\}$,
\begin{align}
S_b(xb^\ell+r) &=  S_b(xb^{\ell-1}+r) + S_b(r); \label{eq:rec_Sb_1}\\
S_b(xb^\ell+xb^{\ell-1}+r) &=  2 S_b(xb^{\ell-1}+r) - S_b(r); \label{eq:rec_Sb_2}\\
S_b(xb^\ell+yb^{\ell-1}+r) &= S_b(xb^{\ell-1}+r) + 2 S_b(yb^{\ell-1}+r) - 2 S_b(r). \label{eq:rec_Sb_3}
\end{align}
\end{prop}

For the sake of completeness, we recall the definition of a particularly useful tool called \textit{the trie of subwords} to prove Proposition~\ref{prop:rel-s}. This tool is also useful to prove the $b$-regularity of the sequence $(S_b(n))_{n\ge 0}$; see Section~\ref{sec:breg}.

\begin{definition}\label{def:trie}
Let $w$ be a finite word over $\{0,\ldots, b-1\}$. 
The language of its subwords is factorial, i.e., if $xyz$ is a subword of $w$, then $y$ is also a subword of $w$. 
Thus we may associate with $w$, the {\em trie\footnote{This tree is also called prefix tree or radix tree. 
All successors of a node have a common prefix and the root is the empty word.} of its subwords}. 
The root is $\varepsilon$ and if $u$ and $ua$ are two subwords of $w$ with $a\in\{0,\ldots, b-1\}$, then $ua$ is a child of $u$. 
We let $\mathcal{T}(w)$ denote the subtree in which we only consider the children $1,\ldots,b-1$ of the root $\varepsilon$ and their successors, if they exist. 
\end{definition}

\begin{remark} 
\label{rem:numberofnodes}
The number of nodes on level $\ell\ge 0$ in $\mathcal{T}(w)$ counts the number of subwords of length $\ell$ in $L_b$ occurring in $w$.
In particular, the number of nodes of the trie $\mathcal{T}(\rep_b(n))$ is exactly $S_b(n)$ for all $n\ge 0$.
\end{remark}

\begin{definition}\label{def:subtree}
For each non-empty word $w \in L_b$, we consider a factorization of $w$ into maximal blocks of consecutively distinct letters (i.e., $a_i \neq a_{i+1}$ for all $i$) of the form
\[
	w = a_1^{n_1} \cdots a_M^{n_M},
\]
with $n_\ell \geq 1$ for all $\ell$.
For each $\ell \in \{0,\dots,M-1\}$, we consider the subtree $T_\ell$ of $\mathcal{T}(w)$ whose root is the node $a_1^{n_1} \cdots a_\ell^{n_\ell} a_{\ell+1}$. For convenience, we set $T_M$ to be an empty tree with no node. Roughly speaking, we have a root of a new subtree $T_\ell$ for each new variation of digits in $w$. For each $\ell \in \{0,\dots,M-1\}$, we also let $\# T_\ell$ denote the number of nodes of the tree $T_\ell$.

Note that for $k-i \ge 2$, one could possibly have $a_k=a_i$.
For each $\ell \in \{0,\dots,M-1\}$, we let $\mathrm{Alph}(\ell)$ denote the set of letters occurring in $a_{\ell+1}\cdots a_M$.
Then for each letter $a \in \mathrm{Alph}(\ell)$, we let $j(a,\ell)$ denote the smallest index in $\{\ell+1,\ldots,M\}$ such that $a_{j(a,\ell)} = a$.
\end{definition}

\begin{example}\label{ex:truc}
In this example, we set $b=3$ and $w = 22000112 \in L_3$.
Using the previous notation, we have $M=4$, $a_1=2$, $a_2=0$, $a_3=1$ and $a_4=2$. For instance, $\mathrm{Alph}(0)=\{0,1,2\}$, $\mathrm{Alph}(2)=\{1,2\}$ and $j(0,0)=2$, $j(1,0)=3$, $j(2,0)=1$ and $j(2,1)=4$.
\end{example}

The following result describes the structure of the tree $\mathcal{T}(w)$. It directly follows from the definition.

\begin{prop}[{\cite[Proposition~27]{LRS2}}]
Let $w$ be a finite word in $L_b$. 
With the above notation about $M$ and the subtrees $T_\ell$, the tree $\mathcal{T}(w)$ has the following properties.
\begin{enumerate}
\item
The node of label $\varepsilon$ has $\#(\mathrm{Alph}(0)\setminus\{0\})$ children that are $a$ for $a \in \mathrm{Alph}(0)\setminus\{0\}$. 
Each child $a$ is the root of a tree isomorphic $T_{j(a,0)-1}$.
\item
For each $\ell \in \{0,\dots,M-1\}$ and each $i \in \{0,\dots,n_{\ell+1}-1\}$ with $(\ell,i) \neq (0,0)$, the node of label $x = a_1^{n_1} \cdots a_{\ell}^{n_{\ell}}a_{\ell+1}^i$ has $\#(\mathrm{Alph}(\ell))$ children that are $xa$ for $a \in \mathrm{Alph}(\ell)$. 
Each child $xa$ with $a \neq a_{\ell+1}$ is the root of a tree isomorphic to $T_{j(a,\ell)-1}$.
\end{enumerate}
\end{prop}

\begin{example}
Let us continue Example~\ref{ex:truc}.
The tree $\mathcal{T}(22000112)$ is depicted in Figure~\ref{fig:trie-base-3}. We use three different colors to represent the letters $0,1,2$.
\begin{figure}[h!tb]
\centering
\includegraphics[trim= 550 500 550 120]{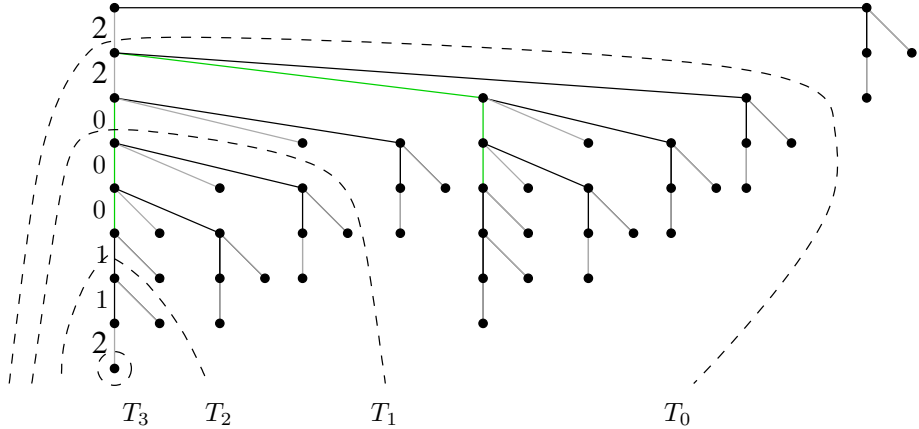}
    \caption{The trie $\mathcal{T}(22000112)$.}
    \label{fig:trie-base-3}
\end{figure}
The tree $T_0$ (resp., $T_1$; resp., $T_2$; resp., $T_3$) is the subtree of $\mathcal{T}(w)$ with root $2$ (resp., $2^20$; resp., $2^20^31$; resp., $2^20^31^22$). 
These subtrees are represented in Figure~\ref{fig:trie-base-3} using dashed lines. 
The tree $T_3$ is limited to a single node since the number of nodes of $T_{M-1}$ is $n_M$, which is equal to $1$ in this example.
\end{example}

Using tries of subwords, we prove the following five lemmas. Their proofs are essentially the same, so we only prove two of them.

\begin{lemma}\label{lem:arbrex00}
For each letter $x\in\{1,\ldots,b-1\}$ and each word $u\in\{0,\ldots,b-1\}^*$, we have
$$\#\left\{ v\in L_b \mid \binom{x00u}{v} > 0 \right\} = 2 \cdot \#\left\{ v\in L_b \mid \binom{x0u}{v} > 0 \right\} - \#\left\{ v\in L_b \mid \binom{xu}{v} > 0 \right\}.$$
\end{lemma}
\begin{proof}
Recall that from Remark~\ref{rem:numberofnodes}, we need to prove that 
$\# \mathcal{T}(x00u) = 2\# \mathcal{T}(x0u) - \# \mathcal{T}(xu)$. 

Assume first that $u$ is of the form $u=0^n$, $n\ge 0$. 
The tree $\mathcal{T}(xu)$ is linear and has $n+2$ nodes, $\mathcal{T}(x0u)$ has $n+3$ nodes and $\mathcal{T}(x00u)$ has $n+4$ nodes. 
The formula holds.

Now suppose that $u$ contains other letters than $0$. 
We let $a_1, \ldots, a_m$ denote all the pairwise distinct letters of $u$ different from $0$. 
They are implicitly ordered with respect to their first appearance in $u$.  
If $x\in\{a_1,\ldots,a_m\}$, we let $i_x\in\{1,\ldots,m\}$ denote the index such that $a_{i_x}=x$.
For all $i\in\{1,\ldots,m\}$, we let $u_i a_i$ denote the prefix of $u$ that ends with the first occurrence of the letter $a_i$ in $u$, and we let $R_i$ denote the subtree of $\mathcal{T}(xu)$ with root $xu_i a_i$.

First, observe that the subtree $T$ of $\mathcal{T}(xu)$ with root $x$ is equal to the subtree of $\mathcal{T}(x0u)$ with root $x0$ and also to the subtree of $\mathcal{T}(x00u)$ with root $x00$. 

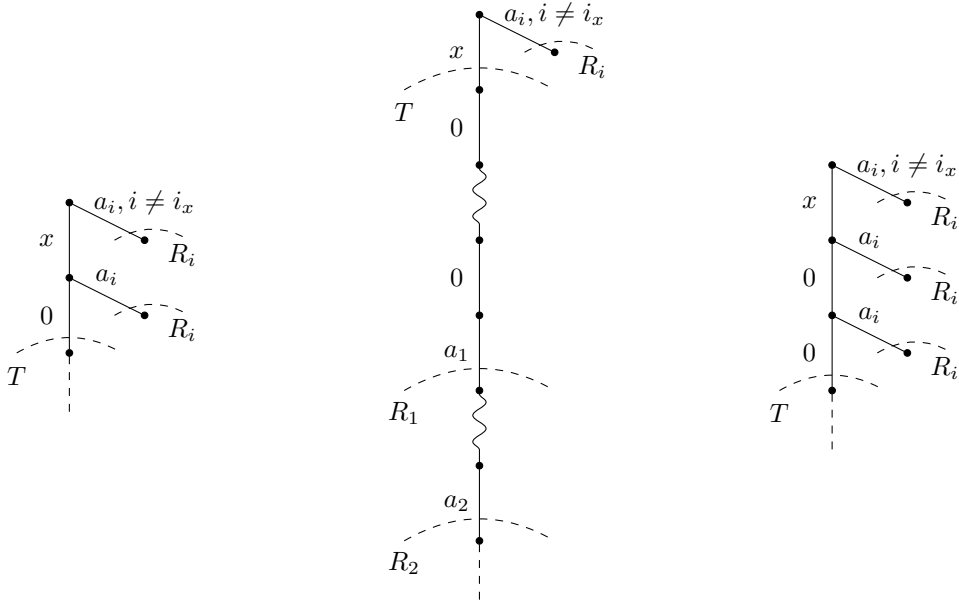
\begin{figure}
\begin{minipage}[c]{0.3\linewidth}
\begin{tikzpicture}
[decoration=snake,
line around/.style={decoration={pre length=#1,post length=#1}}]

  \tikzstyle{every state}=[shape=circle,radius=0.01cm,fill=black,scale=0.1]

  \node[state] 		   (E) at (0,1)        {};
  \node[state] 		   (A) at (0,0)        {};
  \node[state]        (B) at (0,-1)		   {};
  \node[state]        (C) at (1,-0.5)		   {};
  \node[state]        (D) at (1,0.5)		   {};

\tikzstyle{every node}=[shape=circle,fill=none,draw=none,minimum size=10pt,inner sep=2pt]
\node(a1) at (-0.3,0.5) {$x$};
\node(a2) at (-0.3,-0.5) {$0$};
\node(a3) at (-0.7,-1.3) {$T$};
\node(a4) at (0,-2) {};
\node(a5) at (1.5,-0.7) {$R_i$};
\node(a6) at (1.5,0.3) {$R_i$};
\node(a7) at (1,1) {$a_i, i\ne i_x$};
\node(a8) at (0.5,0) {$a_i$};

    \draw[] (E) -- (A);  
    \draw[] (A) -- (B);
    \draw[] (A) -- (C);
    \draw[] (E) -- (D);
    \draw[dashed] (B) -- (a4);

    \draw[dashed] (-0.7,-1) node{} to[bend left] (0.7,-1)  node{};
    \draw[dashed] (0.6,-0.5) node{} to[bend left] (1.6,-0.5)  node{};
   \draw[dashed] (0.6,0.5) node{} to[bend left] (1.6,0.5)  node{};

;
\end{tikzpicture}
\end{minipage}
\begin{minipage}[c]{0.3\linewidth}
\begin{tikzpicture}
[decoration=snake,
line around/.style={decoration={pre length=#1,post length=#1}}]

  \tikzstyle{every state}=[shape=circle,radius=0.01cm,fill=black,scale=0.1]

  \node[state] 		   (E) at (0,1)        {};
  \node[state] 		   (A) at (0,0)        {};
  \node[state]        (B) at (0,-1)		   {};
  \node[state] 		   (C) at (0,-2)        {};
  \node[state] 		   (D) at (0,-3)        {};
  \node[state] 		   (F) at (0,-4)        {};
  \node[state] 		   (G) at (0,-5)        {};
  \node[state] 		   (H) at (0,-6)        {};
  \node[state]        (I) at (1,0.5)		   {};
  
\tikzstyle{every node}=[shape=circle,fill=none,draw=none,minimum size=10pt,inner sep=2pt]
\node(a1) at (-0.3,0.5) {$x$};
\node(a2) at (-0.3,-0.5) {$0$};
\node(a3) at (-0.3,-2.5) {$0$};
\node(a4) at (-0.3,-3.5) {$a_1$};
\node(a5) at (0,-7) {};
\node(a6) at (-0.3,-5.5) {$a_2$};
\node(a7) at (-1,-0.3) {$T$};
\node(a8) at (-1,-4.3) {$R_1$};
\node(a9) at (-1,-6.3) {$R_2$};
\node(a10) at (1.5,0.3) {$R_i$};
\node(a11) at (1,1) {$a_i, i\ne i_x$};

    \draw[] (E) -- (A);  
    \draw[] (A) -- (B);
    \draw[] (C) -- (D);
    \draw[] (D) -- (F);
    \draw[] (G) -- (H);
    \draw[] (E) -- (I);
    \draw[dashed] (H) -- (a5);
    \draw[decorate] (B) -- (C);
    \draw[decorate] (F) -- (G);
    
    \draw[dashed] (-1,0) node{} to[bend left] (1,0)  node{};
    \draw[dashed] (-1,-4) node{} to[bend left] (1,-4)  node{};
    \draw[dashed] (-1,-6) node{} to[bend left] (1,-6)  node{};
    \draw[dashed] (0.6,0.5) node{} to[bend left] (1.6,0.5)  node{};
;
\end{tikzpicture}
\end{minipage}
\begin{minipage}[c]{0.3\linewidth}
\begin{tikzpicture}
[decoration=snake,
line around/.style={decoration={pre length=#1,post length=#1}}]

  \tikzstyle{every state}=[shape=circle,radius=0.01cm,fill=black,scale=0.1]

  \node[state] 		   (E) at (0,1)        {};
  \node[state] 		   (A) at (0,0)        {};
  \node[state]        (B) at (0,-1)		   {};
  \node[state]        (C) at (0,-2)		   {};
  \node[state]        (D) at (1,-0.5)		   {};
  \node[state]        (F) at (1,-1.5)		   {};
  \node[state]        (G) at (1,0.5)		   {};
  
\tikzstyle{every node}=[shape=circle,fill=none,draw=none,minimum size=10pt,inner sep=2pt]
\node(a1) at (-0.3,0.5) {$x$};
\node(a2) at (-0.3,-0.5) {$0$};
\node(a3) at (-0.3,-1.5) {$0$};
\node(a4) at (-0.7,-2.3) {$T$};
\node(a5) at (0,-3) {};
\node(a6) at (1.5,-1.7) {$R_i$};
\node(a7) at (1.5,-0.7) {$R_i$};
\node(a8) at (1.5,0.3) {$R_i$};
\node(a9) at (1,1) {$a_i, i\ne i_x$};
\node(a10) at (0.5,-1) {$a_i$};
\node(a11) at (0.5,0) {$a_i$};

    \draw[](E) -- (A);  
    \draw[] (A) -- (B);
    \draw[] (B) -- (C);
    \draw[](A) -- (D);
    \draw[](B) -- (F);  
    \draw[] (E) -- (G);
    \draw[dashed] (C) -- (a5);

   \draw[dashed] (-0.7,-2) node{} to[bend left] (0.7,-2)  node{};
   \draw[dashed] (0.6,-0.5) node{} to[bend left] (1.6,-0.5)  node{};
   \draw[dashed] (0.6,-1.5) node{} to[bend left] (1.6,-1.5)  node{};
   \draw[dashed] (0.6,0.5) node{} to[bend left] (1.6,0.5)  node{};

;
\end{tikzpicture}
\end{minipage}
\caption{Schematic structure of the trees $\mathcal{T}(x0u)$, $\mathcal{T}(xu)$ and $\mathcal{T}(x00u)$.}
\label{fig:tree_x00}
\end{figure}

Secondly, for all $i\in\{1,\ldots,m\}$, the subtree of $\mathcal{T}(x0u)$ with root $xa_i$ is $R_i$. 
Similarly, $\mathcal{T}(x00u)$ contains two copies of $R_i$: the subtrees of root $xa_i$ and $x0a_i$. 

Finally, for all $i\in\{1,\ldots,m\}$ with $i\ne i_x$, the subtree of $\mathcal{T}(x0u)$ with root $a_i$ is $R_i$ and the subtree of $\mathcal{T}(x00u)$ with root $a_i$ is $R_i$. 

The situation is depicted in Figure~\ref{fig:tree_x00} where we put a unique edge for several indices when necessary, e.g., the edge labeled by $a_i$ stands for $m$ edges labeled by $a_1,\ldots,a_m$. The claimed formula holds since 
$$2\cdot (2+\#T + 2\sum_{\substack{1 \le i \le m \\ i \neq i_x}} \#R_i + \#R_{i_x}) - ( 1+\#T + \sum_{\substack{1 \le i \le m \\ i \neq i_x}} \#R_i )= 3+ \#T + 3 \sum_{\substack{1 \le i \le m \\ i \neq i_x}} \#R_i + 2\#R_{i_x}.$$
\end{proof}

\begin{lemma}\label{lem:arbrexx0}
For each letter $x\in\{1,\ldots,b-1\}$ and each word $u\in\{0,\ldots,b-1\}^*$, we have
$$\#\left\{ v\in L_b \mid \binom{xx0u}{v} > 0 \right\} = \#\left\{ v\in L_b \mid \binom{x0u}{v} > 0 \right\} + \#\left\{ v\in L_b \mid \binom{xu}{v} > 0 \right\}.$$
\end{lemma}
\begin{proof}
The proof is similar to the proof of Lemma~\ref{lem:arbrex00}.
\end{proof}

\begin{lemma}\label{lem:arbrex0y}
For all letters $x,y\in\{1,\ldots,b-1\}$ and each word $u\in\{0,\ldots,b-1\}^*$, we have
$$\#\left\{ v\in L_b \mid \binom{x0yu}{v} > 0 \right\} =  \#\left\{ v\in L_b \mid \binom{xyu}{v} > 0 \right\} + \#\left\{ v\in L_b \mid \binom{yu}{v} > 0 \right\}.$$
\end{lemma}
\begin{proof}
The proof is similar to the proof of Lemma~\ref{lem:arbrex00}. Observe that one needs to divide the proof into two cases according to whether $x$ is equal to $y$ or not. As a first case, also consider $u=y^n$ with $n\ge 0$ instead of $u=0^n$ with $n\ge 0$.
\end{proof}

\begin{lemma}\label{lem:arbrexxy}
For all letters $x,y\in\{1,\ldots,b-1\}$ and each word $u\in\{0,\ldots,b-1\}^*$, we have
$$\#\left\{ v\in L_b \mid \binom{xxyu}{v} > 0 \right\} = 2\cdot  \#\left\{ v\in L_b \mid \binom{xyu}{v} > 0 \right\} - \#\left\{ v\in L_b \mid \binom{yu}{v} > 0 \right\}.$$
\end{lemma}
\begin{proof}
The proof is similar to the proof of Lemma~\ref{lem:arbrex0y}. 
\end{proof}

The next lemma having a slightly more technical proof, we present it.

\begin{lemma}\label{lem:arbrexyz}
For all letters $x,y\in\{1,\ldots,b-1\}$ with $x\ne y$, $z\in\{0,\ldots,b-1\}$ and each word $u\in\{0,\ldots,b-1\}^*$, we have
\begin{align*}
\#\left\{ v\in L_b \mid \binom{xyzu}{v} > 0 \right\} =& \#\left\{ v\in L_b \mid \binom{xzu}{v} > 0 \right\} + 2\cdot \#\left\{ v\in L_b \mid \binom{yzu}{v} > 0 \right\} \\
& - 2\cdot \#\left\{ v\in L_b \mid \binom{zu}{v} > 0 \right\}.
\end{align*}
\end{lemma}
\begin{proof}
Let $x,y\in\{1,\ldots,b-1\}$ with $x\ne y$, $z\in\{0,\ldots,b-1\}$, and let $u\in\{0,\ldots,b-1\}^*$. 
Our reasoning is again based on the structure of the associated trees.
The proof is divided into two cases depending on the fact that $z=0$ or not. 

$\bullet$ As a first case, suppose that $z\ne 0$. 
Now assume that $u$ is of the form $u=z^n$, $n\ge 0$. 
If $x\ne z$ and $y\ne z$, the tree $\mathcal{T}(zu)$ is linear and has $n+2$ nodes, $\mathcal{T}(xzu)$ and $\mathcal{T}(yzu)$ have $2(n+2)$ nodes and $\mathcal{T}(xyzu)$ has $4(n+2)$ nodes and the claimed formula holds. If $x\ne z$ and $y = z$, the tree $\mathcal{T}(zu)$ is linear and has $n+2$ nodes, $\mathcal{T}(xzu)$ has $2(n+2)$ nodes, $\mathcal{T}(yzu)$ has $n+3$ nodes and $\mathcal{T}(xyzu)$ has $2(n+3)$ nodes and the claimed formula holds.
If $x = z$ and $y \ne z$, the tree $\mathcal{T}(zu)$ is linear and has $n+2$ nodes, $\mathcal{T}(xzu)$ has $n+3$ nodes, $\mathcal{T}(yzu)$ has $2(n+2)$ nodes and $\mathcal{T}(xyzu)$ has $3(n+2)+1$ nodes and the claimed formula holds.

Now suppose that $u$ contains other letters than $z$. 
We let $a_1, \ldots, a_m$ denote all the pairwise distinct letters of $u$ different from $z$.
They are implicitly ordered with respect to their first appearance in $u$. 
If $x,y,0\in\{a_1,\ldots,a_m\}$, we let $i_x,i_y,i_0\in\{1,\ldots,m\}$ respectively denote the indices such that $a_{i_x}=x$, $a_{i_y}=y$ and $a_{i_0}=0$.
For all $i\in\{1,\ldots,m\}$, we let $u_i a_i$ denote the prefix of $u$ that ends with the first occurrence of the letter $a_i$ in $u$, and we let $R_i$ denote the subtree of $\mathcal{T}(zu)$ with root $zu_i a_i$.

First, observe that the subtree $T$ of $\mathcal{T}(zu)$ with root $z$ is equal to the subtree of $\mathcal{T}(xzu)$ with root $xz$, to the subtree of $\mathcal{T}(yzu)$ with root $yz$ and also to the subtree of $\mathcal{T}(xyzu)$ with root $xyz$. 

Suppose that $x\ne z$ and $y\ne z$.  
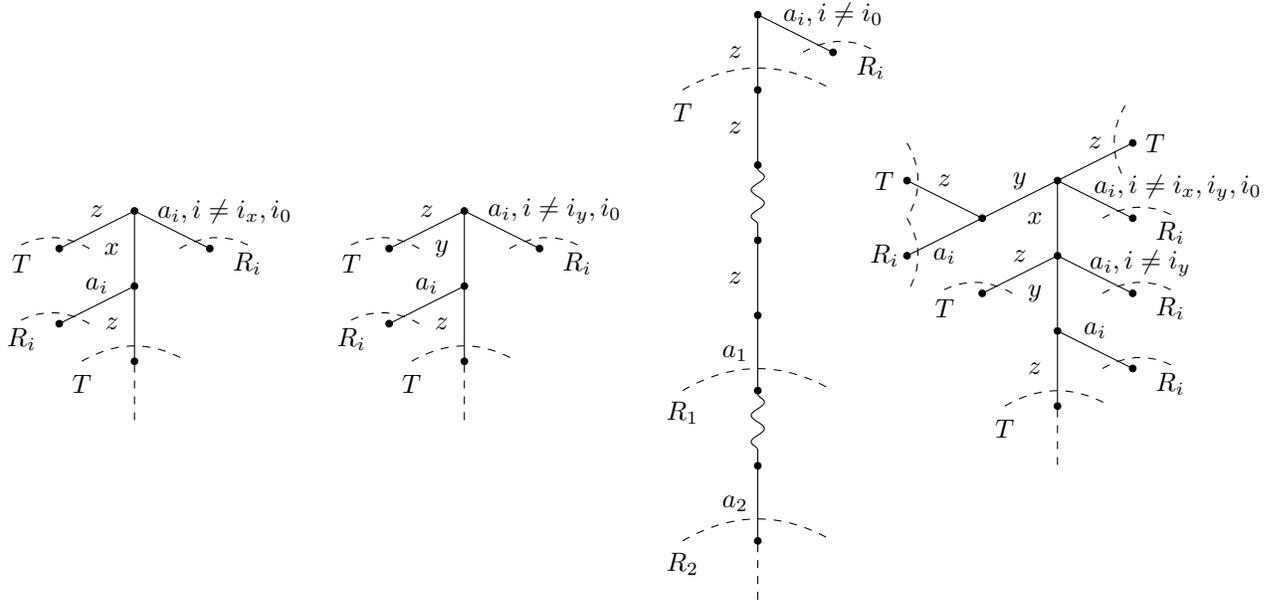
\begin{figure}
\begin{minipage}[c]{0.10\linewidth}
\begin{tikzpicture}
[decoration=snake,
line around/.style={decoration={pre length=#1,post length=#1}}]

  \tikzstyle{every state}=[shape=circle,radius=0.01cm,fill=black,scale=0.1]

  \node[state] 		   (E) at (0,1)        {};
  \node[state] 		   (A) at (0,0)        {};
  \node[state]        (B) at (0,-1)		   {};
  \node[state]        (D) at (1,0.5)		   {};
  \node[state]        (F) at (-1,-0.5)		   {};
  \node[state]        (G) at (-1,0.5)		   {};

\tikzstyle{every node}=[shape=circle,fill=none,draw=none,minimum size=10pt,inner sep=2pt]
\node(a1) at (-0.3,0.5) {$x$};
\node(a2) at (-0.3,-0.5) {$z$};
\node(a3) at (-0.7,-1.3) {$T$};
\node(a4) at (0,-2) {};
\node(a6) at (1.5,0.3) {$R_i$};
\node(a7) at (1.2,1) {$a_i, i\ne i_x, i_0$};
\node(a7) at (-0.5,1) {$z$};
\node(a6) at (-1.5,0.3) {$T$};
\node(a6) at (-1.5,-0.7) {$R_i$};
\node(a1) at (-0.5,0) {$a_i$};

    \draw[] (E) -- (A);  
    \draw[] (A) -- (B);
    \draw[] (E) -- (D);
    \draw[dashed] (B) -- (a4);
    \draw[] (A) -- (F);
    \draw[] (E) -- (G);

    \draw[dashed] (-0.7,-1) node{} to[bend left] (0.7,-1)  node{};
   \draw[dashed] (0.6,0.5) node{} to[bend left] (1.6,0.5)  node{};
   \draw[dashed] (-0.6,0.5) node{} to[bend right] (-1.6,0.5); 
   \draw[dashed] (-0.6,-0.5) node{} to[bend right] (-1.6,-0.5) ;
\end{tikzpicture}
\end{minipage} \hspace{2.5cm}
\begin{minipage}[c]{0.10\linewidth}
\begin{tikzpicture}
[decoration=snake,
line around/.style={decoration={pre length=#1,post length=#1}}]

  \tikzstyle{every state}=[shape=circle,radius=0.01cm,fill=black,scale=0.1]

  \node[state] 		   (E) at (0,1)        {};
  \node[state] 		   (A) at (0,0)        {};
  \node[state]        (B) at (0,-1)		   {};
\node[state]        (D) at (1,0.5)		   {};
  \node[state]        (F) at (-1,-0.5)		   {};
  \node[state]        (G) at (-1,0.5)		   {};

\tikzstyle{every node}=[shape=circle,fill=none,draw=none,minimum size=10pt,inner sep=2pt]
\node(a1) at (-0.3,0.5) {$y$};
\node(a2) at (-0.3,-0.5) {$z$};
\node(a3) at (-0.7,-1.3) {$T$};
\node(a4) at (0,-2) {};
\node(a6) at (1.5,0.3) {$R_i$};
\node(a7) at (1.2,1) {$a_i, i\ne i_y, i_0$};
\node(a7) at (-0.5,1) {$z$};
\node(a6) at (-1.5,0.3) {$T$};
\node(a6) at (-1.5,-0.7) {$R_i$};
\node(a1) at (-0.5,0) {$a_i$};

    \draw[] (E) -- (A);  
    \draw[] (A) -- (B);
    \draw[] (E) -- (D);
    \draw[dashed] (B) -- (a4);
    \draw[] (A) -- (F);
    \draw[] (E) -- (G);

    \draw[dashed] (-0.7,-1) node{} to[bend left] (0.7,-1)  node{};
   \draw[dashed] (0.6,0.5) node{} to[bend left] (1.6,0.5)  node{};
   \draw[dashed] (-0.6,0.5) node{} to[bend right] (-1.6,0.5); 
   \draw[dashed] (-0.6,-0.5) node{} to[bend right] (-1.6,-0.5) ;
\end{tikzpicture}
\end{minipage} \hspace{2.5cm}
\begin{minipage}[c]{0.005\linewidth}
\begin{tikzpicture}
[decoration=snake,
line around/.style={decoration={pre length=#1,post length=#1}}]

  \tikzstyle{every state}=[shape=circle,radius=0.01cm,fill=black,scale=0.1]

  \node[state] 		   (E) at (0,1)        {};
  \node[state] 		   (A) at (0,0)        {};
  \node[state]        (B) at (0,-1)		   {};
  \node[state] 		   (C) at (0,-2)        {};
  \node[state] 		   (D) at (0,-3)        {};
  \node[state] 		   (F) at (0,-4)        {};
  \node[state] 		   (G) at (0,-5)        {};
  \node[state] 		   (H) at (0,-6)        {};
  \node[state]        (I) at (1,0.5)		   {};
  
\tikzstyle{every node}=[shape=circle,fill=none,draw=none,minimum size=10pt,inner sep=2pt]
\node(a1) at (-0.3,0.5) {$z$};
\node(a2) at (-0.3,-0.5) {$z$};
\node(a3) at (-0.3,-2.5) {$z$};
\node(a4) at (-0.3,-3.5) {$a_1$};
\node(a5) at (0,-7) {};
\node(a6) at (-0.3,-5.5) {$a_2$};
\node(a7) at (-1,-0.3) {$T$};
\node(a8) at (-1,-4.3) {$R_1$};
\node(a9) at (-1,-6.3) {$R_2$};
\node(a10) at (1.5,0.3) {$R_i$};
\node(a11) at (1,1) {$a_i, i\ne i_0$};

    \draw[] (E) -- (A);  
    \draw[] (A) -- (B);
    \draw[] (C) -- (D);
    \draw[] (D) -- (F);
    \draw[] (G) -- (H);
    \draw[] (E) -- (I);
    \draw[dashed] (H) -- (a5);
    \draw[decorate] (B) -- (C);
    \draw[decorate] (F) -- (G);
    
    \draw[dashed] (-1,0) node{} to[bend left] (1,0)  node{};
    \draw[dashed] (-1,-4) node{} to[bend left] (1,-4)  node{};
    \draw[dashed] (-1,-6) node{} to[bend left] (1,-6)  node{};
    \draw[dashed] (0.6,0.5) node{} to[bend left] (1.6,0.5)  node{};
;
\end{tikzpicture}
\end{minipage}\hspace{2.5cm}
\begin{minipage}[c]{0.10\linewidth}
\begin{tikzpicture}
[decoration=snake,
line around/.style={decoration={pre length=#1,post length=#1}}]

  \tikzstyle{every state}=[shape=circle,radius=0.01cm,fill=black,scale=0.1]

  \node[state] 		   (E) at (0,1)        {};
  \node[state] 		   (A) at (0,0)        {};
  \node[state]        (B) at (0,-1)		   {};
  \node[state]        (C) at (0,-2)		   {};
  \node[state]        (D) at (1,-0.5)		   {};
  \node[state]        (F) at (1,-1.5)		   {};
  \node[state]        (G) at (1,0.5)		   {};
  \node[state]        (H) at (-1,-0.5)		   {};        
  \node[state]        (I) at (-1,0.5)		   {};
  \node[state]        (J) at (1,1.5)		   {};
  \node[state] 		  (K) at (-2,1)        {};
  \node[state] 		  (L) at (-2,0)        {};

\tikzstyle{every node}=[shape=circle,fill=none,draw=none,minimum size=10pt,inner sep=2pt]
\node(a1) at (-0.3,0.5) {$x$};
\node(a2) at (-0.3,-0.5) {$y$};
\node(a3) at (-0.3,-1.5) {$z$};
\node(a4) at (-0.7,-2.3) {$T$};
\node(a5) at (0,-3) {};
\node(a6) at (1.5,-1.7) {$R_i$};
\node(a7) at (1.5,-0.7) {$R_i$};
\node(a7) at (-1.5,-0.7) {$T$};
\node(a8) at (1.5,0.3) {$R_i$};
\node(a9) at (1.6,0.9) {$a_i, i\ne i_x, i_y, i_0$};
\node(a10) at (0.5,-1) {$a_i$};
\node(a11) at (1.1,-0.1) {$a_i, i\ne i_y$};
\node(a9) at (0.5,1.5) {$z$};
\node(a9) at (1.3,1.5) {$T$};
\node(a9) at (-0.5,1) {$y$};
\node(a9) at (-0.5,0) {$z$};
\node(a9) at (-1.5,0) {$a_i$};
\node(a9) at (-1.5,1) {$z$};
\node(a9) at (-2.3,0) {$R_i$};
\node(a9) at (-2.3,1) {$T$};

    \draw[](E) -- (A);  
    \draw[] (A) -- (B);
    \draw[] (B) -- (C);
    \draw[](A) -- (D);
    \draw[](B) -- (F);  
    \draw[] (E) -- (G);
    \draw[](A) -- (H);
    \draw[](E) -- (I);
    \draw[](E) -- (J); 
    \draw[](I) -- (K);
    \draw[](I) -- (L);  
    \draw[dashed] (C) -- (a5);

   \draw[dashed] (-0.7,-2) node{} to[bend left] (0.7,-2)  node{};
   \draw[dashed] (0.6,-0.5) node{} to[bend left] (1.6,-0.5)  node{};
     \draw[dashed] (-0.6,-0.5) node{} to[bend right] (-1.6,-0.5)  node{};
   \draw[dashed] (0.6,-1.5) node{} to[bend left] (1.6,-1.5)  node{};
   \draw[dashed] (0.6,0.5) node{} to[bend left] (1.6,0.5)  node{};
   \draw[dashed] (0.9,2) node{} to[bend right] (0.9,1)  node{};
   \draw[dashed] (-2,1.5) node{} to[bend left] (-2,0.5)  node{};
   \draw[dashed] (-2,0.5) node{} to[bend left] (-2,-0.5)  node{};
;
\end{tikzpicture}
\end{minipage}
\caption{Schematic structure of the trees $\mathcal{T}(xzu)$, $\mathcal{T}(yzu)$, $\mathcal{T}(zu)$ and $\mathcal{T}(xyzu)$ when $x\ne z$, $y\ne z$ and $z\ne 0$.}
\label{fig:tree_xyz1}
\end{figure}
Using the same reasoning as in the proof of Lemma~\ref{lem:arbrex00}, the situation is depicted in Figure~\ref{fig:tree_xyz1}. 
The claimed formula holds since 
\begin{align*}
&(2+2\#T + 2\sum_{\substack{1 \le i \le m \\ i \neq i_x, i_y, i_0}} \#R_i + \#R_{i_x} + 2 \#R_{i_y} + \#R_{i_0}) \\
&+2\cdot (2+2\#T + 2\sum_{\substack{1 \le i \le m \\ i \neq i_x, i_y, i_0}} \#R_i + 2 \#R_{i_x} + \#R_{i_y} + \#R_{i_0}) \\
&- 2\cdot (1+\#T + \sum_{\substack{1 \le i \le m \\ i \neq i_x, i_y, i_0}} \#R_i + \#R_{i_x} + \#R_{i_y}) \\
&= 4+4\#T + 4\sum_{\substack{1 \le i \le m \\ i \neq i_x, i_y, i_0}} \#R_i + 3\#R_{i_x} + 2 \#R_{i_y} + 3\#R_{i_0}.
\end{align*}

Suppose that $x\ne z$ and $y= z$.  
\begin{figure}
\begin{minipage}[c]{0.10\linewidth}
\begin{tikzpicture}
[decoration=snake,
line around/.style={decoration={pre length=#1,post length=#1}}]

  \tikzstyle{every state}=[shape=circle,radius=0.01cm,fill=black,scale=0.1]

  \node[state] 		   (E) at (0,1)        {};
  \node[state] 		   (A) at (0,0)        {};
  \node[state]        (B) at (0,-1)		   {};
  \node[state]        (D) at (1,0.5)		   {};
  \node[state]        (F) at (-1,-0.5)		   {};
  \node[state]        (G) at (-1,0.5)		   {};

\tikzstyle{every node}=[shape=circle,fill=none,draw=none,minimum size=10pt,inner sep=2pt]
\node(a1) at (-0.3,0.5) {$x$};
\node(a2) at (-0.3,-0.5) {$z$};
\node(a3) at (-0.7,-1.3) {$T$};
\node(a4) at (0,-2) {};
\node(a6) at (1.5,0.3) {$R_i$};
\node(a7) at (1.2,1) {$a_i, i\ne i_x, i_0$};
\node(a7) at (-0.5,1) {$z$};
\node(a6) at (-1.5,0.3) {$T$};
\node(a6) at (-1.5,-0.7) {$R_i$};
\node(a1) at (-0.5,0) {$a_i$};

    \draw[] (E) -- (A);  
    \draw[] (A) -- (B);
    \draw[] (E) -- (D);
    \draw[dashed] (B) -- (a4);
    \draw[] (A) -- (F);
    \draw[] (E) -- (G);

    \draw[dashed] (-0.7,-1) node{} to[bend left] (0.7,-1)  node{};
   \draw[dashed] (0.6,0.5) node{} to[bend left] (1.6,0.5)  node{};
   \draw[dashed] (-0.6,0.5) node{} to[bend right] (-1.6,0.5); 
   \draw[dashed] (-0.6,-0.5) node{} to[bend right] (-1.6,-0.5) ;
\end{tikzpicture}
\end{minipage} \hspace{2.5cm}
\begin{minipage}[c]{0.10\linewidth}
\begin{tikzpicture}
[decoration=snake,
line around/.style={decoration={pre length=#1,post length=#1}}]

  \tikzstyle{every state}=[shape=circle,radius=0.01cm,fill=black,scale=0.1]

  \node[state] 		   (E) at (0,1)        {};
  \node[state] 		   (A) at (0,0)        {};
  \node[state]        (B) at (0,-1)		   {};
  \node[state]        (D) at (1,0.5)		   {};
  \node[state]        (F) at (-1,-0.5)		   {};

\tikzstyle{every node}=[shape=circle,fill=none,draw=none,minimum size=10pt,inner sep=2pt]
\node(a1) at (-0.3,0.5) {$z$};
\node(a2) at (-0.3,-0.5) {$z$};
\node(a3) at (-0.7,-1.3) {$T$};
\node(a4) at (0,-2) {};
\node(a6) at (1.5,0.3) {$R_i$};
\node(a7) at (1,1) {$a_i, i\ne i_0$};
\node(a6) at (-1.5,-0.7) {$R_i$};
\node(a1) at (-0.5,0) {$a_i$};

    \draw[] (E) -- (A);  
    \draw[] (A) -- (B);
    \draw[] (E) -- (D);
    \draw[dashed] (B) -- (a4);
    \draw[] (A) -- (F);

    \draw[dashed] (-0.7,-1) node{} to[bend left] (0.7,-1)  node{};
   \draw[dashed] (0.6,0.5) node{} to[bend left] (1.6,0.5)  node{};
   \draw[dashed] (-0.6,-0.5) node{} to[bend right] (-1.6,-0.5) ;
\end{tikzpicture}
\end{minipage} \hspace{2.5cm}
\begin{minipage}[c]{0.005\linewidth}
\begin{tikzpicture}
[decoration=snake,
line around/.style={decoration={pre length=#1,post length=#1}}]

  \tikzstyle{every state}=[shape=circle,radius=0.01cm,fill=black,scale=0.1]

  \node[state] 		   (E) at (0,1)        {};
  \node[state] 		   (A) at (0,0)        {};
  \node[state]        (B) at (0,-1)		   {};
  \node[state] 		   (C) at (0,-2)        {};
  \node[state] 		   (D) at (0,-3)        {};
  \node[state] 		   (F) at (0,-4)        {};
  \node[state] 		   (G) at (0,-5)        {};
  \node[state] 		   (H) at (0,-6)        {};
  \node[state]        (I) at (1,0.5)		   {};
  
\tikzstyle{every node}=[shape=circle,fill=none,draw=none,minimum size=10pt,inner sep=2pt]
\node(a1) at (-0.3,0.5) {$z$};
\node(a2) at (-0.3,-0.5) {$z$};
\node(a3) at (-0.3,-2.5) {$z$};
\node(a4) at (-0.3,-3.5) {$a_1$};
\node(a5) at (0,-7) {};
\node(a6) at (-0.3,-5.5) {$a_2$};
\node(a7) at (-1,-0.3) {$T$};
\node(a8) at (-1,-4.3) {$R_1$};
\node(a9) at (-1,-6.3) {$R_2$};
\node(a10) at (1.5,0.3) {$R_i$};
\node(a11) at (1,1) {$a_i, i\ne i_0$};

    \draw[] (E) -- (A);  
    \draw[] (A) -- (B);
    \draw[] (C) -- (D);
    \draw[] (D) -- (F);
    \draw[] (G) -- (H);
    \draw[] (E) -- (I);
    \draw[dashed] (H) -- (a5);
    \draw[decorate] (B) -- (C);
    \draw[decorate] (F) -- (G);
    
    \draw[dashed] (-1,0) node{} to[bend left] (1,0)  node{};
    \draw[dashed] (-1,-4) node{} to[bend left] (1,-4)  node{};
    \draw[dashed] (-1,-6) node{} to[bend left] (1,-6)  node{};
    \draw[dashed] (0.6,0.5) node{} to[bend left] (1.6,0.5)  node{};
;
\end{tikzpicture}
\end{minipage} \hspace{2.5cm}
\begin{minipage}[c]{0.10\linewidth}
\begin{tikzpicture}
[decoration=snake,
line around/.style={decoration={pre length=#1,post length=#1}}]

  \tikzstyle{every state}=[shape=circle,radius=0.01cm,fill=black,scale=0.1]

  \node[state] 		   (E) at (0,1)        {};
  \node[state] 		   (A) at (0,0)        {};
  \node[state]        (B) at (0,-1)		   {};
  \node[state]        (C) at (0,-2)		   {};
  \node[state]        (D) at (1,-0.5)		   {};
  \node[state]        (F) at (1,-1.5)		   {};
  \node[state]        (G) at (1,0.5)		   {};
  \node[state]        (I) at (-1,0.5)		   {};
  \node[state] 		  (K) at (-2,1)        {};
  \node[state] 		  (L) at (-2,0)        {};

\tikzstyle{every node}=[shape=circle,fill=none,draw=none,minimum size=10pt,inner sep=2pt]
\node(a1) at (-0.3,0.5) {$x$};
\node(a2) at (-0.3,-0.5) {$z$};
\node(a3) at (-0.3,-1.5) {$z$};
\node(a4) at (-0.7,-2.3) {$T$};
\node(a5) at (0,-3) {};
\node(a6) at (1.5,-1.7) {$R_i$};
\node(a7) at (1.5,-0.7) {$R_i$};
\node(a8) at (1.5,0.3) {$R_i$};
\node(a9) at (1.4,0.9) {$a_i, i\ne i_x, i_0$};
\node(a10) at (0.5,-1) {$a_i$};
\node(a11) at (0.5,-0.1) {$a_i$};
\node(a9) at (-0.5,1) {$z$};
\node(a9) at (-1.5,0) {$a_i$};
\node(a9) at (-1.5,1) {$z$};
\node(a9) at (-2.3,0) {$R_i$};
\node(a9) at (-2.3,1) {$T$};

    \draw[](E) -- (A);  
    \draw[] (A) -- (B);
    \draw[] (B) -- (C);
    \draw[](A) -- (D);
    \draw[](B) -- (F);  
    \draw[] (E) -- (G);
    \draw[](E) -- (I);
    \draw[](I) -- (K);
    \draw[](I) -- (L);  
    \draw[dashed] (C) -- (a5);

   \draw[dashed] (-0.7,-2) node{} to[bend left] (0.7,-2)  node{};
   \draw[dashed] (0.6,-0.5) node{} to[bend left] (1.6,-0.5)  node{};
   \draw[dashed] (0.6,-1.5) node{} to[bend left] (1.6,-1.5)  node{};
   \draw[dashed] (0.6,0.5) node{} to[bend left] (1.6,0.5)  node{};
   \draw[dashed] (-2,1.5) node{} to[bend left] (-2,0.5)  node{};
   \draw[dashed] (-2,0.5) node{} to[bend left] (-2,-0.5)  node{};
;
\end{tikzpicture}
\end{minipage}
\caption{Schematic structure of the trees $\mathcal{T}(xzu)$, $\mathcal{T}(yzu)$, $\mathcal{T}(zu)$ and $\mathcal{T}(xyzu)$ when $x\ne z$, $y= z$ and $z\ne 0$.}
\label{fig:tree_xyz2}
\end{figure}
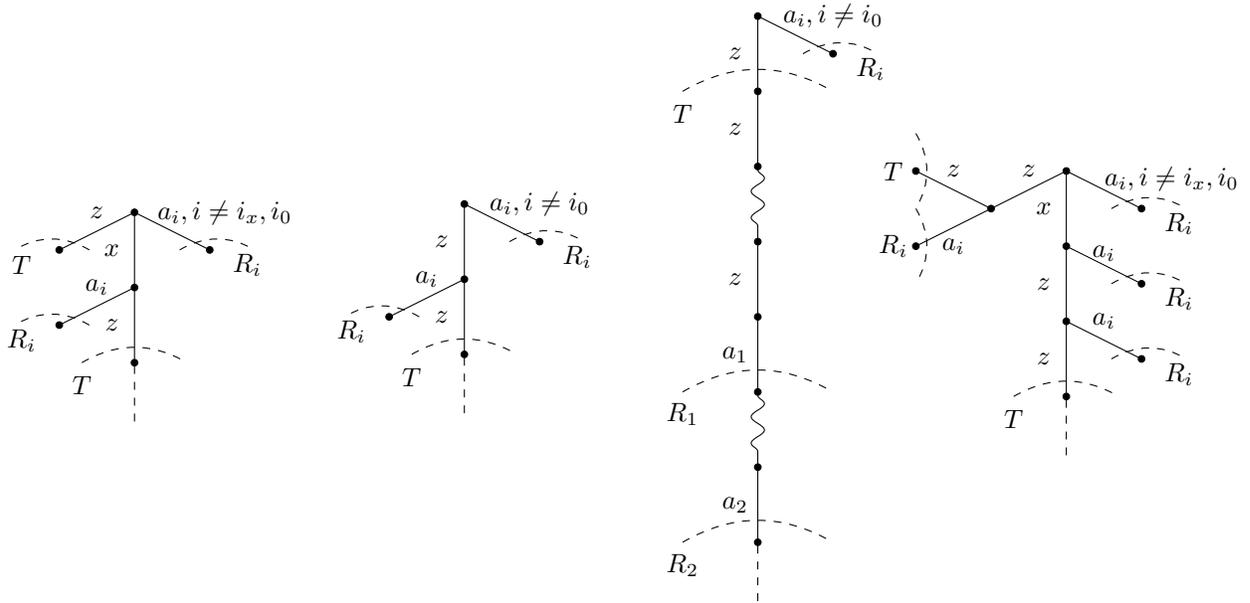
The situation is depicted in Figure~\ref{fig:tree_xyz2}. The claimed formula holds since 
\begin{align*}
&(2+2\#T + 2\sum_{\substack{1 \le i \le m \\ i \neq i_x, i_0}} \#R_i + \#R_{i_x} + \#R_{i_0}) \\
&+2\cdot (2+\#T + 2\sum_{\substack{1 \le i \le m \\ i \neq i_x, i_0}} \#R_i + 2 \#R_{i_x} + \#R_{i_0}) \\
&- 2\cdot (1+\#T + \sum_{\substack{1 \le i \le m \\ i \neq i_x, i_0}} \#R_i + \#R_{i_x} ) \\
&= 4+2\#T + 4\sum_{\substack{1 \le i \le m \\ i \neq i_x, i_0}} \#R_i + 3\#R_{i_x} + 3\#R_{i_0}
.
\end{align*}

Suppose that $x= z$ and $y\ne z$.
\begin{figure}
\begin{minipage}[c]{0.10\linewidth}
\begin{tikzpicture}
[decoration=snake,
line around/.style={decoration={pre length=#1,post length=#1}}]

  \tikzstyle{every state}=[shape=circle,radius=0.01cm,fill=black,scale=0.1]

  \node[state] 		   (E) at (0,1)        {};
  \node[state] 		   (A) at (0,0)        {};
  \node[state]        (B) at (0,-1)		   {};
  \node[state]        (D) at (1,0.5)		   {};
  \node[state]        (F) at (-1,-0.5)		   {};
  
\tikzstyle{every node}=[shape=circle,fill=none,draw=none,minimum size=10pt,inner sep=2pt]
\node(a1) at (-0.3,0.5) {$z$};
\node(a2) at (-0.3,-0.5) {$z$};
\node(a3) at (-0.7,-1.3) {$T$};
\node(a4) at (0,-2) {};
\node(a6) at (1.5,0.3) {$R_i$};
\node(a7) at (1,1) {$a_i, i\ne  i_0$};
\node(a6) at (-1.5,-0.7) {$R_i$};
\node(a1) at (-0.5,0) {$a_i$};

    \draw[] (E) -- (A);  
    \draw[] (A) -- (B);
    \draw[] (E) -- (D);
    \draw[dashed] (B) -- (a4);
    \draw[] (A) -- (F);

    \draw[dashed] (-0.7,-1) node{} to[bend left] (0.7,-1)  node{};
   \draw[dashed] (0.6,0.5) node{} to[bend left] (1.6,0.5)  node{};
   \draw[dashed] (-0.6,-0.5) node{} to[bend right] (-1.6,-0.5) ;
\end{tikzpicture}
\end{minipage} \hspace{2.5cm}
\begin{minipage}[c]{0.10\linewidth}
\begin{tikzpicture}
[decoration=snake,
line around/.style={decoration={pre length=#1,post length=#1}}]

  \tikzstyle{every state}=[shape=circle,radius=0.01cm,fill=black,scale=0.1]

  \node[state] 		   (E) at (0,1)        {};
  \node[state] 		   (A) at (0,0)        {};
  \node[state]        (B) at (0,-1)		   {};
  \node[state]        (D) at (1,0.5)		   {};
  \node[state]        (F) at (-1,-0.5)		   {};
  \node[state]        (G) at (-1,0.5)		   {};

\tikzstyle{every node}=[shape=circle,fill=none,draw=none,minimum size=10pt,inner sep=2pt]
\node(a1) at (-0.3,0.5) {$y$};
\node(a2) at (-0.3,-0.5) {$z$};
\node(a3) at (-0.7,-1.3) {$T$};
\node(a4) at (0,-2) {};
\node(a6) at (1.5,0.3) {$R_i$};
\node(a7) at (1.2,1) {$a_i, i\ne i_y, i_0$};
\node(a7) at (-0.5,1) {$z$};
\node(a6) at (-1.5,0.3) {$T$};
\node(a6) at (-1.5,-0.7) {$R_i$};
\node(a1) at (-0.5,0) {$a_i$};

    \draw[] (E) -- (A);  
    \draw[] (A) -- (B);
    \draw[] (E) -- (D);
    \draw[dashed] (B) -- (a4);
    \draw[] (A) -- (F);
    \draw[] (E) -- (G);

    \draw[dashed] (-0.7,-1) node{} to[bend left] (0.7,-1)  node{};
   \draw[dashed] (0.6,0.5) node{} to[bend left] (1.6,0.5)  node{};
   \draw[dashed] (-0.6,0.5) node{} to[bend right] (-1.6,0.5); 
   \draw[dashed] (-0.6,-0.5) node{} to[bend right] (-1.6,-0.5) ;
\end{tikzpicture}
\end{minipage} \hspace{2.5cm}
\begin{minipage}[c]{0.005\linewidth}
\begin{tikzpicture}
[decoration=snake,
line around/.style={decoration={pre length=#1,post length=#1}}]

  \tikzstyle{every state}=[shape=circle,radius=0.01cm,fill=black,scale=0.1]

  \node[state] 		   (E) at (0,1)        {};
  \node[state] 		   (A) at (0,0)        {};
  \node[state]        (B) at (0,-1)		   {};
  \node[state] 		   (C) at (0,-2)        {};
  \node[state] 		   (D) at (0,-3)        {};
  \node[state] 		   (F) at (0,-4)        {};
  \node[state] 		   (G) at (0,-5)        {};
  \node[state] 		   (H) at (0,-6)        {};
  \node[state]        (I) at (1,0.5)		   {};

\tikzstyle{every node}=[shape=circle,fill=none,draw=none,minimum size=10pt,inner sep=2pt]
\node(a1) at (-0.3,0.5) {$z$};
\node(a2) at (-0.3,-0.5) {$z$};
\node(a3) at (-0.3,-2.5) {$z$};
\node(a4) at (-0.3,-3.5) {$a_1$};
\node(a5) at (0,-7) {};
\node(a6) at (-0.3,-5.5) {$a_2$};
\node(a7) at (-1,-0.3) {$T$};
\node(a8) at (-1,-4.3) {$R_1$};
\node(a9) at (-1,-6.3) {$R_2$};
\node(a10) at (1.5,0.3) {$R_i$};
\node(a11) at (1,1) {$a_i, i\ne i_0$};

    \draw[] (E) -- (A);  
    \draw[] (A) -- (B);
    \draw[] (C) -- (D);
    \draw[] (D) -- (F);
    \draw[] (G) -- (H);
    \draw[] (E) -- (I);
    \draw[dashed] (H) -- (a5);
    \draw[decorate] (B) -- (C);
    \draw[decorate] (F) -- (G);
    
    \draw[dashed] (-1,0) node{} to[bend left] (1,0)  node{};
    \draw[dashed] (-1,-4) node{} to[bend left] (1,-4)  node{};
    \draw[dashed] (-1,-6) node{} to[bend left] (1,-6)  node{};
    \draw[dashed] (0.6,0.5) node{} to[bend left] (1.6,0.5)  node{};
;
\end{tikzpicture}
\end{minipage} \hspace{2.5cm}
\begin{minipage}[c]{0.10\linewidth}
\begin{tikzpicture}
[decoration=snake,
line around/.style={decoration={pre length=#1,post length=#1}}]

  \tikzstyle{every state}=[shape=circle,radius=0.01cm,fill=black,scale=0.1]

  \node[state] 		   (E) at (0,1)        {};
  \node[state] 		   (A) at (0,0)        {};
  \node[state]        (B) at (0,-1)		   {};
  \node[state]        (C) at (0,-2)		   {};
  \node[state]        (D) at (1,-0.5)		   {};
  \node[state]        (F) at (1,-1.5)		   {};
  \node[state]        (G) at (1,0.5)		   {};
  \node[state]        (I) at (-1,0.5)		   {};
  \node[state] 		  (K) at (-2,1)        {};
  \node[state] 		  (L) at (-2,0)        {};
  \node[state]        (M) at (-1,-0.5)		   {};

\tikzstyle{every node}=[shape=circle,fill=none,draw=none,minimum size=10pt,inner sep=2pt]
\node(a1) at (-0.3,0.5) {$z$};
\node(a2) at (-0.3,-0.5) {$y$};
\node(a3) at (-0.3,-1.5) {$z$};
\node(a4) at (-0.7,-2.3) {$T$};
\node(a5) at (0,-3) {};
\node(a6) at (1.5,-1.7) {$R_i$};
\node(a7) at (1.5,-0.7) {$R_i$};
\node(a8) at (1.5,0.3) {$R_i$};
\node(a9) at (1.4,0.9) {$a_i, i\ne i_y, i_0$};
\node(a10) at (0.5,-1) {$a_i$};
\node(a11) at (1.3,-0.1) {$a_i, i\ne i_y$};
\node(a11) at (-0.5,-0.1) {$z$};
\node(a9) at (-0.5,1) {$y$};
\node(a9) at (-1.5,0) {$a_i$};
\node(a9) at (-1.5,1) {$z$};
\node(a9) at (-2.3,0) {$R_i$};
\node(a9) at (-2.3,1) {$T$};
\node(a8) at (-1.5,-0.7) {$T$};

    \draw[](E) -- (A);  
    \draw[] (A) -- (B);
    \draw[] (B) -- (C);
    \draw[](A) -- (D);
    \draw[](B) -- (F);  
    \draw[] (E) -- (G);
    \draw[](E) -- (I);
    \draw[](I) -- (K);
    \draw[](I) -- (L);
    \draw[] (A) -- (M);  
    \draw[dashed] (C) -- (a5);

   \draw[dashed] (-0.7,-2) node{} to[bend left] (0.7,-2)  node{};
   \draw[dashed] (0.6,-0.5) node{} to[bend left] (1.6,-0.5)  node{};
   \draw[dashed] (0.6,-1.5) node{} to[bend left] (1.6,-1.5)  node{};
   \draw[dashed] (0.6,0.5) node{} to[bend left] (1.6,0.5)  node{};
    \draw[dashed] (-0.6,-0.5) node{} to[bend right] (-1.6,-0.5)  node{};
   \draw[dashed] (-2,1.5) node{} to[bend left] (-2,0.5)  node{};
   \draw[dashed] (-2,0.5) node{} to[bend left] (-2,-0.5)  node{};
;
\end{tikzpicture}
\end{minipage}
\caption{Schematic structure of the trees $\mathcal{T}(xzu)$, $\mathcal{T}(yzu)$, $\mathcal{T}(zu)$ and $\mathcal{T}(xyzu)$ when $x= z$, $y\ne  z$ and $z\ne 0$.}
\label{fig:tree_xyz3}
\end{figure}
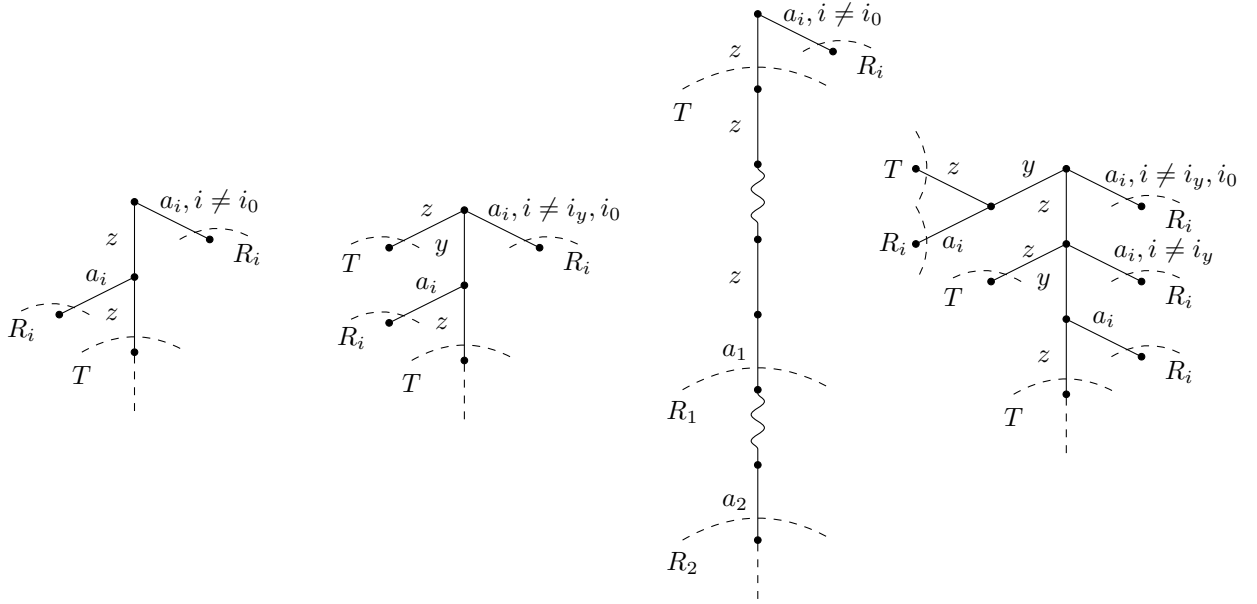  
The situation is depicted in Figure~\ref{fig:tree_xyz3}. The claimed formula holds since 
\begin{align*}
&(2+\#T + 2\sum_{\substack{1 \le i \le m \\ i \neq i_y, i_0}} \#R_i + 2 \#R_{i_y} + \#R_{i_0}) \\
&+2\cdot (2+2\#T + 2\sum_{\substack{1 \le i \le m \\ i \neq i_y, i_0}} \#R_i + \#R_{i_y} + \#R_{i_0}) \\
&- 2\cdot (1+\#T + \sum_{\substack{1 \le i \le m \\ i \neq i_y, i_0}} \#R_i +\#R_{i_y}) \\
&= 4+3\#T + 4\sum_{\substack{1 \le i \le m \\ i \neq i_y, i_0}} \#R_i + 2 \#R_{i_y} + 3\#R_{i_0}
.
\end{align*}

$\bullet$ As a second case, suppose that $z=0$. 
Then, by convention, leading zeroes are not allowed in base-$b$ expansions and we must prove that the following formula holds
\begin{align*}
\#\left\{ v\in L_b \mid \binom{xy0u}{v} > 0 \right\} =& \#\left\{ v\in L_b \mid \binom{x0u}{v} > 0 \right\} + 2\cdot \#\left\{ v\in L_b \mid \binom{y0u}{v} > 0 \right\} \\
& - 2\cdot \#\left\{ v\in L_b \mid \binom{\rep_b(\val_b(u))}{v} > 0 \right\}.
\end{align*}
It is useful to note that $\rep_b(\val_b(\cdot)):\{0,\ldots,b-1\}^* \mapsto L_b$ plays a normalization role. It removes leading zeroes. 

If $u=0^n$, with $n\ge 0$, then $\rep_b(\val_b(u))=\varepsilon$ and the tree $\mathcal{T}(\rep_b(\val_b(u)))$ has only one node. 
The trees $\mathcal{T}(x0u)$ and $\mathcal{T}(y0u)$ both have $n+3$ nodes and the tree $\mathcal{T}(xy0u)$ has $3(n+2)+1$ nodes and the claimed formula holds. 

Now suppose that $u$ contains other letters than $0$. 
We let $a_1, \ldots, a_m$ denote all the pairwise distinct letters of $u$ different from $0$. 
They are implicitly ordered with respect to their first appearance in $u$.
If $x,y\in\{a_1,\ldots,a_m\}$, we let $i_x,i_y\in\{1,\ldots,m\}$ respectively denote the indices such that $a_{i_x}=x$ and $a_{i_y}=y$.
For all $i\in\{1,\ldots,m\}$, we let $u'_i a_i$ denote the prefix of $\rep_b(\val_b(u))$ that ends with the first occurrence of the letter $a_i$ in $\rep_b(\val_b(u))$, and we let $R_i$ denote the subtree of $\mathcal{T}(\rep_b(\val_b(u)))$ with root $u'_i a_i$.

\begin{figure}
\begin{minipage}[c]{0.10\linewidth}
\begin{tikzpicture}
[decoration=snake,
line around/.style={decoration={pre length=#1,post length=#1}}]

  \tikzstyle{every state}=[shape=circle,radius=0.01cm,fill=black,scale=0.1]

  \node[state] 		   (E) at (0,1)        {};
  \node[state] 		   (A) at (0,0)        {};
  \node[state]        (B) at (0,-1)		   {};
  \node[state]        (D) at (1,0.5)		   {};
  \node[state]        (F) at (-1,-0.5)		   {};

\tikzstyle{every node}=[shape=circle,fill=none,draw=none,minimum size=10pt,inner sep=2pt]
\node(a1) at (-0.3,0.5) {$x$};
\node(a2) at (-0.3,-0.5) {$0$};
\node(a3) at (-0.7,-1.3) {$T$};
\node(a4) at (0,-2) {};
\node(a6) at (1.5,0.3) {$R_i$};
\node(a7) at (1,1) {$a_i, i\ne i_x$};
\node(a6) at (-1.5,-0.7) {$R_i$};
\node(a1) at (-0.5,0) {$a_i$};

    \draw[] (E) -- (A);  
    \draw[] (A) -- (B);
    \draw[] (E) -- (D);
    \draw[dashed] (B) -- (a4);
    \draw[] (A) -- (F);

    \draw[dashed] (-0.7,-1) node{} to[bend left] (0.7,-1)  node{};
   \draw[dashed] (0.6,0.5) node{} to[bend left] (1.6,0.5)  node{};
   \draw[dashed] (-0.6,-0.5) node{} to[bend right] (-1.6,-0.5) ;
\end{tikzpicture}
\end{minipage} \hspace{2.5cm}
\begin{minipage}[c]{0.10\linewidth}
\begin{tikzpicture}
[decoration=snake,
line around/.style={decoration={pre length=#1,post length=#1}}]

  \tikzstyle{every state}=[shape=circle,radius=0.01cm,fill=black,scale=0.1]

  \node[state] 		   (E) at (0,1)        {};
  \node[state] 		   (A) at (0,0)        {};
  \node[state]        (B) at (0,-1)		   {};
  \node[state]        (D) at (1,0.5)		   {};
  \node[state]        (F) at (-1,-0.5)		   {};

\tikzstyle{every node}=[shape=circle,fill=none,draw=none,minimum size=10pt,inner sep=2pt]
\node(a1) at (-0.3,0.5) {$y$};
\node(a2) at (-0.3,-0.5) {$0$};
\node(a3) at (-0.7,-1.3) {$T$};
\node(a4) at (0,-2) {};
\node(a6) at (1.5,0.3) {$R_i$};
\node(a7) at (1,1) {$a_i, i\ne i_y$};
\node(a6) at (-1.5,-0.7) {$R_i$};
\node(a1) at (-0.5,0) {$a_i$};

    \draw[] (E) -- (A);  
    \draw[] (A) -- (B);
    \draw[] (E) -- (D);
    \draw[dashed] (B) -- (a4);
    \draw[] (A) -- (F);

    \draw[dashed] (-0.7,-1) node{} to[bend left] (0.7,-1)  node{};
   \draw[dashed] (0.6,0.5) node{} to[bend left] (1.6,0.5)  node{};
   \draw[dashed] (-0.6,-0.5) node{} to[bend right] (-1.6,-0.5) ;
\end{tikzpicture}
\end{minipage} \hspace{2.5cm}
\begin{minipage}[c]{0.005\linewidth}
\begin{tikzpicture}
[decoration=snake,
line around/.style={decoration={pre length=#1,post length=#1}}]

  \tikzstyle{every state}=[shape=circle,radius=0.01cm,fill=black,scale=0.1]

  \node[state] 		   (E) at (0,1)        {};
  \node[state] 		   (A) at (0,0)        {};
  \node[state]        (B) at (0,-1)		   {};
  \node[state] 		   (C) at (0,-2)        {};
  \node[state] 		   (D) at (0,-3)        {};
  \node[state] 		   (F) at (0,-4)        {};
  \node[state]        (I) at (1,0.5)		   {};
  
\tikzstyle{every node}=[shape=circle,fill=none,draw=none,minimum size=10pt,inner sep=2pt]
\node(a1) at (-0.3,0.5) {$a_1$};
\node(a2) at (-0.3,-0.5) {$a_1$};
\node(a3) at (-0.3,-2.5) {$a_1$};
\node(a4) at (-0.3,-3.5) {$a_2$};
\node(a5) at (0,-5) {};
\node(a7) at (-1,-0.3) {$R_1$};
\node(a8) at (-1,-4.3) {$R_2$};
\node(a10) at (1.5,0.3) {$R_i$};
\node(a11) at (1,1) {$a_i, i\ne 1$};

    \draw[] (E) -- (A);  
    \draw[] (A) -- (B);
    \draw[] (C) -- (D);
    \draw[] (D) -- (F);
    \draw[] (E) -- (I);
    \draw[dashed] (F) -- (a5);
    \draw[decorate] (B) -- (C);
    
    \draw[dashed] (-1,0) node{} to[bend left] (1,0)  node{};
    \draw[dashed] (-1,-4) node{} to[bend left] (1,-4)  node{};
    \draw[dashed] (0.6,0.5) node{} to[bend left] (1.6,0.5)  node{};
;
\end{tikzpicture}
\end{minipage} \hspace{2.5cm}
\begin{minipage}[c]{0.10\linewidth}
\begin{tikzpicture}
[decoration=snake,
line around/.style={decoration={pre length=#1,post length=#1}}]

  \tikzstyle{every state}=[shape=circle,radius=0.01cm,fill=black,scale=0.1]

  \node[state] 		   (E) at (0,1)        {};
  \node[state] 		   (A) at (0,0)        {};
  \node[state]        (B) at (0,-1)		   {};
  \node[state]        (C) at (0,-2)		   {};
  \node[state]        (D) at (1,-0.5)		   {};
  \node[state]        (F) at (1,-1.5)		   {};
  \node[state]        (G) at (1,0.5)		   {};
  \node[state]        (I) at (-1,0.5)		   {};
  \node[state] 		  (K) at (-2,1)        {};
  \node[state] 		  (L) at (-2,0)        {};
  \node[state]        (M) at (-1,-0.5)		   {};

\tikzstyle{every node}=[shape=circle,fill=none,draw=none,minimum size=10pt,inner sep=2pt]
\node(a1) at (-0.3,0.5) {$x$};
\node(a2) at (-0.3,-0.5) {$y$};
\node(a3) at (-0.3,-1.5) {$0$};
\node(a4) at (-0.7,-2.3) {$T$};
\node(a5) at (0,-3) {};
\node(a6) at (1.5,-1.7) {$R_i$};
\node(a7) at (1.5,-0.7) {$R_i$};
\node(a8) at (1.5,0.3) {$R_i$};
\node(a9) at (1.4,0.9) {$a_i, i\ne i_x, i_y$};
\node(a10) at (0.5,-1) {$a_i$};
\node(a11) at (1.3,-0.1) {$a_i, i\ne i_y$};
\node(a11) at (-0.6,-0.1) {$0$};
\node(a9) at (-0.5,1) {$y$};
\node(a9) at (-1.5,0) {$a_i$};
\node(a9) at (-1.5,1) {$0$};
\node(a9) at (-2.3,0) {$R_i$};
\node(a9) at (-2.3,1) {$T$};
\node(a8) at (-1.5,-0.7) {$T$};

    \draw[](E) -- (A);  
    \draw[] (A) -- (B);
    \draw[] (B) -- (C);
    \draw[](A) -- (D);
    \draw[](B) -- (F);  
    \draw[] (E) -- (G);
    \draw[](E) -- (I);
    \draw[](I) -- (K);
    \draw[](I) -- (L);
    \draw[] (A) -- (M);  
    \draw[dashed] (C) -- (a5);

   \draw[dashed] (-0.7,-2) node{} to[bend left] (0.7,-2)  node{};
   \draw[dashed] (0.6,-0.5) node{} to[bend left] (1.6,-0.5)  node{};
   \draw[dashed] (0.6,-1.5) node{} to[bend left] (1.6,-1.5)  node{};
   \draw[dashed] (0.6,0.5) node{} to[bend left] (1.6,0.5)  node{};
    \draw[dashed] (-0.6,-0.5) node{} to[bend right] (-1.6,-0.5)  node{};
   \draw[dashed] (-2,1.5) node{} to[bend left] (-2,0.5)  node{};
   \draw[dashed] (-2,0.5) node{} to[bend left] (-2,-0.5)  node{};
;
\end{tikzpicture}
\end{minipage}
\caption{Schematic structure of the trees $\mathcal{T}(x0u)$, $\mathcal{T}(y0u)$, $\mathcal{T}(\rep_b(\val_b(u)))$ and $\mathcal{T}(xy0u)$.}
\label{fig:tree_xyz4}
\end{figure}
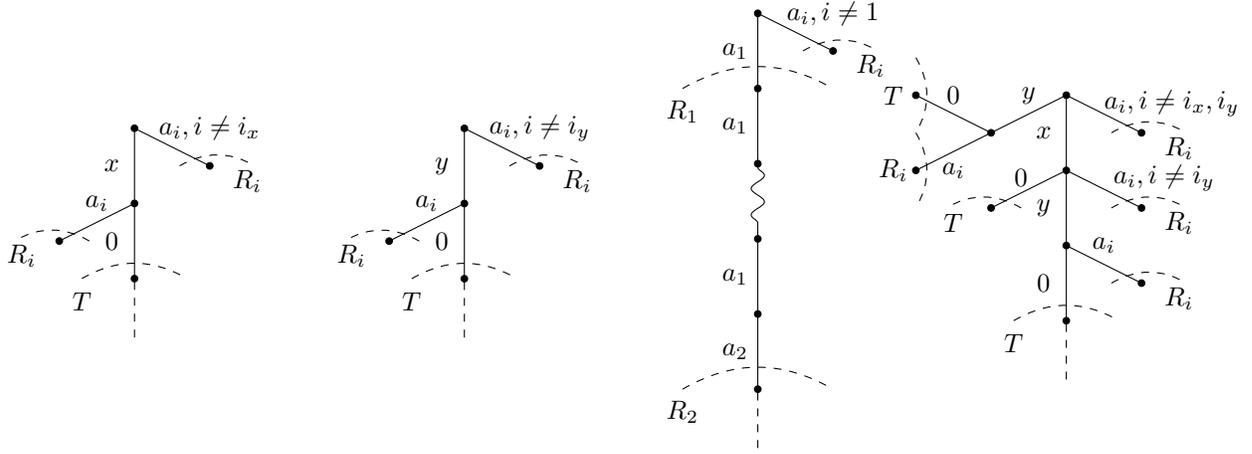  
The situation is depicted in Figure~\ref{fig:tree_xyz4}. Observe that the subtree $T$ of $\mathcal{T}(y0u)$ with root $y0$ is equal to the subtree of $\mathcal{T}(x0u)$ with root $x0$ and to the subtree of $\mathcal{T}(xy0u)$ with root $xy0$. 
The claimed formula holds since 
\begin{align*}
&(2+\#T + 2\sum_{\substack{1 \le i \le m \\ i \neq i_x, i_y}} \#R_i + \#R_{i_x} + 2 \#R_{i_y}) \\
&+2\cdot (2+\#T + 2\sum_{\substack{1 \le i \le m \\ i \neq i_x, i_y}} \#R_i + 2 \#R_{i_x} + \#R_{i_y}) \\
&- 2\cdot (1 + \sum_{\substack{1 \le i \le m \\ i \neq i_x, i_y}} \#R_i + \#R_{i_x} +\#R_{i_y}) \\
&= 4+3\#T + 4\sum_{\substack{1 \le i \le m \\ i \neq i_x, i_y}} \#R_i + 3 \#R_{i_x} + 2\#R_{i_y}.
\end{align*}
\end{proof}

Those five lemmas can be translated into recurrence relations satisfied by the sequence $(S_b(n))_{n\ge 0}$ using Definition~\ref{def:S_b}. 

\begin{proof}[Proof of Proposition~\ref{prop:rel-s}]
The first part is clear using Table~\ref{tab:initialSb}. 
Let $x,y\in\{1,\ldots,b-1\}$ with $x\ne y$. Proceed by induction on $\ell\ge 1$.

Let us first prove~\eqref{eq:rec_Sb_1}. 
If $\ell=1$, then $r=0$ and~\eqref{eq:rec_Sb_1} follows from Table~\ref{tab:initialSb}. 
Now suppose that $\ell\ge 2$ and assume that~\eqref{eq:rec_Sb_1} holds for all $\ell'< \ell$. 
Let $r\in\{0,\ldots,b^{\ell-1}-1\}$, and let $u$ be a word in $\{0,\ldots,b-1\}^*$ such that $|u|\ge 1$ and $\rep_b(xb^\ell+r)=x0u$. 
The proof is divided into two parts according to the first letter of $u$. 
If $u=0u'$ with $u'\in\{0,\ldots,b-1\}^*$, then 
$$
\begin{array}{rclr}
	S_b(xb^\ell +r)
	&=& 2S_b(xb^{\ell-1}+r) - S_b(xb^{\ell-2}+r) & \text{(by Lemma~\ref{lem:arbrex00})}\\
	&=& 2(S_b(xb^{\ell-2}+r) + S_b(r)) - S_b(xb^{\ell-2}+r) & \text{(by induction hypothesis)}\\
	&=& S_b(xb^{\ell-2}+r) + S_b(r)  + S_b(r)& \\
	&=& S_b(xb^{\ell-1}+r) + S_b(r), & \text{(by induction hypothesis)}\\
\end{array}
$$
which proves~\eqref{eq:rec_Sb_1}. 
Now if $u=zu'$ with $z\in\{1,\ldots,b-1\}$ and $u'\in\{0,\ldots,b-1\}^*$, then~\eqref{eq:rec_Sb_1} directly follows from Definition~\ref{def:S_b} and Lemma~\ref{lem:arbrex0y}.

Let us prove~\eqref{eq:rec_Sb_2}. 
If $\ell=1$, then $r=0$ and~\eqref{eq:rec_Sb_1} follows from Table~\ref{tab:initialSb}. 
Now suppose that $\ell\ge 2$ and assume that~\eqref{eq:rec_Sb_2} holds for all $\ell'< \ell$. 
Let $r\in\{0,\ldots,b^{\ell-1}-1\}$, and let $u$ be a word in $\{0,\ldots,b-1\}^*$ such that $|u|\ge 1$ and $\rep_b(xb^\ell+xb^{\ell-1}+r)=xxu$. 
The proof is divided into two parts according to the first letter of $u$. 
If $u=0u'$ with $u'\in\{0,\ldots,b-1\}^*$, then 
$$
\begin{array}{rclr}
	S_b(xb^\ell+xb^{\ell-1} +r)
	&=& S_b(xb^{\ell-1}+r) + S_b(xb^{\ell-2}+r) & \text{(by Lemma~\ref{lem:arbrexx0})}\\
	&=& S_b(xb^{\ell-2}+r) + S_b(r)) + S_b(xb^{\ell-2}+r) & \text{(using~\eqref{eq:rec_Sb_1})}\\
	&=& 2(S_b(xb^{\ell-2}+r) + S_b(r)) - S_b(r) & \\
	&=& 2 S_b(xb^{\ell-1}+r) - S_b(r), & \text{(using~\eqref{eq:rec_Sb_1})}\\
\end{array}
$$
which proves~\eqref{eq:rec_Sb_2}. 
Now if $u=zu'$ with  $z\in\{1,\ldots,b-1\}$ and $u'\in\{0,\ldots,b-1\}^*$, then~\eqref{eq:rec_Sb_2} directly follows from Definition~\ref{def:S_b} and Lemma~\ref{lem:arbrexxy}.

Let us finally prove~\eqref{eq:rec_Sb_3}. 
If $\ell=1$, then $r=0$ and~\eqref{eq:rec_Sb_1} follows from Table~\ref{tab:initialSb}.  
Now suppose that $\ell\ge 2$ and assume that~\eqref{eq:rec_Sb_3} holds for all $\ell'< \ell$. 
Let $r\in\{0,\ldots,b^{\ell-1}-1\}$, let $z$ be a letter of $\{1,\ldots,b-1\}$ and let $u$ be a word in $\{0,\ldots,b-1\}^*$ such that $\rep_b(xb^\ell+yb^{\ell-1}+r)=xyzu$. 
Using Definition~\ref{def:S_b} and Lemma~\ref{lem:arbrexyz}, we directly have that
$$
S_b(xb^\ell+yb^{\ell-1}+r) = S_b(xb^{\ell-1}+r) + S_b(yb^{\ell-1}+r) - 2 S_b(r),
$$
which proves~\eqref{eq:rec_Sb_3}.
\end{proof}

\section{Regularity of the sequence $(S_b(n))_{n\ge 0}$}\label{sec:breg}

The sequence $(S_2(n))_{n\ge 0}$ is shown to be $2$-regular; see \cite{LRS2}. 
We recall that the {\em $b$-kernel} of a sequence $s=(s(n))_{n\geq 0}$ is the set
$$\mathcal{K}_b(s) = \{ (s(b^in + j))_{n\ge 0} | \; i\ge 0 \text{ and } 0\le j < b^i \}.
$$
A sequence $s=(s(n))_{n\geq 0}\in\mathbb{Z}^\mathbb{N}$ is \emph{$b$-regular} if there exists a finite number of sequences $(t_1(n))_{n\geq 0}$, \ldots, $(t_\ell(n))_{n\geq 0}$ such that every sequence in the $\mathbb{Z}$-module $\langle \mathcal{K}_b(s)  \rangle$ generated by the $b$-kernel $\mathcal{K}_b(s)$ is a $\mathbb{Z}$-linear combination of the $t_r$'s. 
In this section, we prove that the sequence $(S_b(n))_{n\ge 0}$ is $b$-regular.
As a consequence, one can get matrices to compute $S_b(n)$ in a number of matrix multiplications proportional to $\log_b(n)$. To prove the $b$-regularity of the sequence $(S_b(n))_{n\geq 0}$ for any base $b$, we first need a lemma involving some matrix manipulations. 

\begin{lemma}\label{lem:matrices}
Let $I$ and $0$ respectively be the identity matrix of size $b^2 \times b^2$ and the zero matrix of size $b^2 \times b^2$. 
Let $M_b$ be the block-matrix of size $b^3\times b^3$
$$
M_b := 
\left( 
\begin{array}{ccccccc}
I & I & 2 I & \cdots & \cdots & \cdots & 2 I\\
2I & 3I & 3 I & 4 I & \cdots & \cdots & 4 I\\
\vdots & \vdots & 4 I & \ddots & \ddots & & \vdots\\
\vdots & \vdots & \vdots & \ddots & \ddots & \ddots & \vdots\\
\vdots & \vdots & \vdots & & \ddots & \ddots & 4I\\
\vdots & \vdots & \vdots & & & \ddots & 3I \\
2I & 3I & 4 I & \cdots & \cdots & \cdots &4I\\
\end{array}
\right).
$$
This matrix is invertible and its inverse is given by
$$
M^{-1}_b := 
\left( 
\begin{array}{ccccccc}
3I & 2I & \cdots & \cdots & 2 I & -(2b-3) I\\
-2I & 0 & \cdots & \cdots & 0 &  I\\
0 & -I & \ddots & & \vdots & \vdots\\
\vdots & \ddots & \ddots & \ddots & \vdots & \vdots\\
\vdots & & \ddots & \ddots & 0 & \vdots\\
0 & \cdots & \cdots & 0 & -I & I\\
\end{array}
\right).
$$
\end{lemma}

For the proof of the previous lemma, simply proceed to the multiplication of the two matrices. Using this lemma, we prove that the sequence $(S_b(n))_{n\ge 0}$ is $b$-regular. 

\begin{theorem}\label{thm:b-reg}
For all $r\in\{0,\ldots, b^2-1\}$, 
we have 
\begin{equation}\label{eq:b-reg}
S_b(nb^2 +r) = a_r S_b(n) + \sum_{s=0}^{b-2} c_{r,s} S_b(nb+s) \quad \forall \, n \ge 0,
\end{equation}
where the coefficients $a_r$ and $c_{r,s}$ are unambiguously determined by the first few values $S_b(0)$, $S_b(1)$,\ldots, $S_b(b^3-1)$ and given in Table~\ref{tab:ar}, Table~\ref{tab:cr0} and Table~\ref{tab:crs}.
In particular, the sequence $(S_b(n))_{n\ge 0}$ is $b$-regular. 
Moreover, a choice of generators for $\langle \mathcal{K}_b(s) \rangle$ is given by the $b$ sequences $(S_b(n))_{n\ge 0}$, $(S_b(bn))_{n\ge 0}$, $(S_b(bn+1))_{n\ge 0}$, \ldots, $(S_b(bn+b-2))_{n\ge 0}$.
\end{theorem}

\begin{table}
\centering
$
\begin{array}{|c|c|c|c|c|c|c|c|c|c|c|}
\hline
\rep_b(r) & \varepsilon & x & b-1 & x0 & (b-1)0 & xx & (b-1)(b-1) & xy & (b-1)x & x(b-1) \\
\hline
a_r & -1 & -2 & 2b-3 & -2 & 4b-4 & -1 & 4b-3 & -2 & 4b-4 & 2b-3 \\
\hline
\end{array}
$
\caption{Values of $a_r$ for $0 \leq r < b^2$ with $x,y \in \{1,\dots,b-2\}$ and $x \neq y$.}
\label{tab:ar}
\end{table}

\begin{table}
\centering
$
\begin{array}{|c|c|c|c|c|c|c|c|c|c|c|}
\hline
\rep_b(r) & \varepsilon & x & b-1 & x0 & (b-1)0 & xx & (b-1)(b-1) & xy & (b-1)x & x(b-1) \\
\hline
c_{r,0} & 2 & 2 & 1 & 1 & -1 & 0 & -2 & 0 & -2 & -1 \\
\hline
\end{array}
$
\caption{Values of $c_{r,0}$ for $0 \leq r < b^2$ with $x,y \in \{1,\dots,b-2\}$ and $x \neq y$.}
\label{tab:cr0}
\end{table}

\begin{table}
\centering
\begin{tabular}{|c|c|c|c|c|c|c|c|c|c|}
\hline
$\rep_b(r)$ & $\varepsilon$ & \multicolumn{2}{c|}{$x$} & $b-1$ & \multicolumn{2}{c|}{$x0$} & $(b-1)0$ & \multicolumn{2}{c|}{$xx$} \\
\hline
$s$ & $z$ & $x$ & $z$ & $z$ & $x$ & $z$ & $z$ & $x$ & $z$ \\
\hline
$c_{r,s}$ & $0$ & $1$ & $0$ & $-1$ & $2$ & $0$ & $-2$ & $2$ & $0$ \\
\hline 
\multicolumn{10}{c}{} \\
\hline
$\rep_b(r)$ & $(b-1)(b-1)$ & \multicolumn{3}{c|}{$xy$} & \multicolumn{2}{c|}{$x(b-1)$} & \multicolumn{3}{c|}{$(b-1)x$} \\
\hline
$s$ & $z$ & $x$ & $y$ & $z$ & $x$ & $z$ & $x$ & \multicolumn{2}{c|}{$z$}  \\
\hline
$c_{r,s}$ & $-2$ & $2$ & $1$ & $0$ & $1$ & $-1$ & $-1$ & \multicolumn{2}{c|}{$-2$} \\
\hline
\end{tabular}
\caption{Values of $c_{r,s}$ for $0 \leq r < b^2$ and $1 \leq s \leq b-2$ with $x,y,z \in \{1,\dots,b-2\}$ pairwise distinct.}
\label{tab:crs}
\end{table}

\begin{proof}
We proceed by induction on $n\ge 0$. 
For the base case $n \in \{0,1,\dots,b^2-1\}$, we first compute the coefficients $a_r$ and $c_{r,s}$ using the values of $S_b(nb^2+r)$ for $n\in\{0,\ldots,b-1\}$ and $r \in \{0,\dots,b^2-1\}$. 
Then we show that~\eqref{eq:b-reg} also holds with these coefficients for $n\in\{b,\ldots,b^2-1\}$.

\textbf{Base case.} 
Let $I$ denote the identity matrix of size $b^2\times b^2$. 
The system of $b^3$ equations~\eqref{eq:b-reg} when $n\in\{0,\ldots,b-1\}$ and $r \in \{0,\dots,b^2-1\}$ can be written as $MX=V$ where the matrix $M\in \mathbb{Z}^{b^3}_{b^3}$ is equal to
$$
\left( 
\begin{array}{ccccccc}
S_b(0) I & S_b(0) I & S_b(1) I & S_b(2) I & \cdots & S_b(b-2) I\\
S_b(1) I & S_b(b) I & S_b(b+1) I & S_b(b+2) I & \cdots & S_b(2b-2) I\\
\vdots & \vdots & \vdots & \vdots &  & \vdots\\
S_b(b-1) I & S_b(b(b-1)) I & S_b(b(b-1)+1) I & S_b(b(b-1)+2) I & \cdots & S_b(b(b-1)+b-2) I\\
\end{array}
\right) 
$$
and the vectors $X,V\in \mathbb{Z}^{b^3}$ are respectively given by
\begin{align*}
X^{\mathsf{T}} &= \left( 
\begin{array}{cccccccccccc}
a_0&\cdots&a_{b^2-1}&c_{0,0}&c_{1,0}&\cdots&c_{b^2-1,0}&
\cdots&c_{0,b-2}&c_{1,b-2}&\cdots&c_{b^2-1,b-2}\\
\end{array}
\right), \\
V^\mathsf{T} &= \left( 
\begin{array}{cccc}
S_b(0) & S_b(1) & \cdots &S_b(b^3-1)\\
\end{array}.
\right)
\end{align*}
Observe that in the vector $X$, the coefficients $c_{r,s}$ are first sorted by $s$ then by $r$.
Using Table~\ref{tab:initialSb}, the matrix $M$ is equal to the matrix $M_b$ of Lemma~\ref{lem:matrices}. 
By this lemma, the previous system has a unique solution given by $X=M^{-1}_b V$. 
Consequently, using Lemma~\ref{lem:matrices}, we have, for all $r\in\{0,\ldots, b^2-1\}$ and all $s\in\{1,\ldots,b-2\}$,
\begin{align*}
a_r &= 3S_b(r) + 2 \sum_{j=1}^{b-2} S_b(jb^2+r) - (2b-3) \, S_b((b-1)b^2+r), \\
c_{r,0} &= -2 S_b(r) + S_b((b-1)b^2+r),\\
c_{r,s} &= -S_b(sb^2 + r) + S_b((b-1)b^2+r).
\end{align*}
The values of the coefficients can then be computed using Table~\ref{tab:initialSb} and are stored in Table~\ref{tab:ar}, Table~\ref{tab:cr0} and Table~\ref{tab:crs}.

For $n \in \{b,\dots,b^2-1\}$, the values of $S_b(nb^2+r)$ are given in Table~\ref{tab:Sbnx0}, Table~\ref{tab:Sbnxx} and Table~\ref{tab:Sbnxy} according to whether $\rep_b(n)$ is of the form $x0$, $xx$ or $xy$ with $x \neq y$.
The proof that~\eqref{eq:b-reg} holds for each $n\in\{b,\ldots,b^2-1\}$ only requires easy computations that are left to the reader.

\begin{table}
\centering
$
\begin{array}{|c|c|c|c|c|c|c|c|c|c|c|}
\hline
\rep_b(r) 
& \varepsilon & x & y & x0 & y0 & xx & yy & xy & yx & yz  
\\
\hline
S_b(nb^2+r) 
& 5 & 7 & 8 & 8 & 10 & 7 & 9 & 10 & 11 & 12
\\
\hline
\end{array}
$
\caption{Values of $S_b(nb^2+r)$ for $b \leq n < b^2$ with $\rep_b(n)=x0$ and $x,y,z \in \{1,\dots,b-1\}$ pairwise distinct.}
\label{tab:Sbnx0}
\end{table}

\begin{table}
\centering
$
\begin{array}{|c|c|c|c|c|c|c|c|c|c|c|}
\hline
\rep_b(r) 
& \varepsilon & x & y & x0 & y0 & xx & yy & xy & yx & yz  
\\
\hline
S_b(nb^2+r) 
& 7 & 8 & 10 & 7 & 11 & 5 & 9 & 8 & 10 & 12
\\
\hline
\end{array}
$
\caption{Values of $S_b(nb^2+r)$ for $b \leq n < b^2$ with $\rep_b(n)=xx$ and $x,y,z \in \{1,\dots,b-1\}$ pairwise distinct.}
\label{tab:Sbnxx}
\end{table}

\begin{table}
\centering
$
\begin{array}{|c|c|c|c|c|c|c|c|c|c|c|c|c|c|c|c|c|c|}
\hline
\rep_b(r) 
& \varepsilon & x & y & z & x0 & y0 & z0 & xx & yy & zz & xy & xz & yx & yz & zx & zy & zt 
\\
\hline
S_b(nb^2+r) 
& 10 & 13 & 12 & 14 & 13 & 11 & 15 & 10 & 8 & 12 & 12 & 14 & 11 & 12 & 15 & 14 & 16 
\\
\hline
\end{array}
$
\caption{Values of $S_b(nb^2+r)$ for $b \leq n < b^2$ with $\rep_b(n)=xy$ and $x,y,z,t \in \{1,\dots,b-1\}$ pairwise distinct.}
\label{tab:Sbnxy}
\end{table}

\textbf{Inductive step.} 
Consider $n\ge b^2$ and suppose that the relation~\eqref{eq:b-reg} holds for all $m < n$. 
Then $|\rep_b(n)|\ge 3$. 
Like for the base case, we need to consider several cases according to the form of the base-$b$ expansion of $n$.
More precisely, we need to consider the following five forms, where $u \in \{0,\dots,b-1\}^*$, $x,y,z \in \{1,\dots,b-1\}$, $x \neq z$, and $t \in \{0,\dots,b-1\}$:
\[
	x00u \quad \text{or} \quad
	xx0u \quad \text{or} \quad
	x0yu \quad \text{or} \quad
	xxyu \quad \text{or} \quad
	xztu. 
\]

Let us focus on the first form of $\rep_b(n)$ since the same reasoning can be applied for the other ones. 
Assume that $\rep_b(n)=x00u$ where $x\in\{1,\ldots,b-1\}$ and $u\in\{0,\ldots,b-1\}^*$. 
For all $r\in\{0,\ldots,b^2-1\}$, there exist $r_1,r_2\in\{0,\ldots,b-1\}$ such that $\val_b(r_1r_2)=r$. We have
$$
\begin{array}{rclr}
	S_b(nb^2 +r)
	&=& S_b(\val_b(x00ur_1r_2)) & \\
	&=& 2 S_b(\val_b(x0ur_1r_2)) - S_b(\val_b(xur_1r_2)) & \text{(by Lemma~\ref{lem:arbrex00})} \\
	&=& a_r \, 2 S_b(\val_b(x0u)) + \sum_{s=0}^{b-2} c_{r,s} \, 2 S_b(\val_b(x0us)) &  \\
	&&- a_r S_b(\val_b(xu)) - \sum_{s=0}^{b-2} c_{r,s} S_b(\val_b(xus)) &  \text{(by induction hypothesis)}\\
	&=& a_r S_b(\val_b(x00u)) +  \sum_{s=0}^{b-2} c_{r,s} S_b(\val_b(x00us)) & \text{(by Lemma~\ref{lem:arbrex00})} \\
	&=& a_r S_b(n) +  \sum_{s=0}^{b-2} c_{r,s} S_b(nb+s), & \text{(by Lemma~\ref{lem:arbrex00})}
\end{array}
$$
which proves~\eqref{eq:b-reg}.

\textbf{$b$-regularity.} From the first part of the proof, we directly deduce that the $\mathbb{Z}$-module $\left\langle \mathcal{K}_b(S_b)\right\rangle$ is generated by the $(b+1)$ sequences 
$$(S_b(n))_{n\ge 0}, (S_b(bn))_{n\ge 0}, (S_b(bn+1))_{n\ge 0}, \ldots, (S_b(bn+b-1))_{n\ge 0}.$$
We now show that we can reduce the number of generators. 
To that aim, we prove that 
\begin{equation}\label{eq:b-reg-2}
S_b(nb + b-1) = (2b-1) S_b(n) - \sum_{s=0}^{b-2} S_b(nb+s) \quad \forall \, n \ge 0.
\end{equation}
We proceed by induction on $n\ge 0$. 
As a base case, the proof that~\eqref{eq:b-reg-2} holds for each $n\in\{b,\ldots,b^2-1\}$ only requires easy computations that are left to the reader (using Table~\ref{tab:initialSb}).
Now consider $n\ge b^2$ and suppose that the relation~\eqref{eq:b-reg-2} holds for all $m < n$. 
Then $|\rep_b(n)|\ge 3$. 
Mimicking the first induction step of this proof, we need to consider several cases according to the form of the base-$b$ expansion of $n$.
More precisely, we need to consider the following five forms, where $u \in \{0,\dots,b-1\}^*$, $x,y,z \in \{1,\dots,b-1\}$, $x \neq z$, and $t \in \{0,\dots,b-1\}$:
\[
	x00u \quad \text{or} \quad
	xx0u \quad \text{or} \quad
	x0yu \quad \text{or} \quad
	xxyu \quad \text{or} \quad
	xztu. 
\]
Let us focus on the first form of $\rep_b(n)$ since the same reasoning can be applied for the other ones. 
Assume that $\rep_b(n)=x00u$ where $x\in\{1,\ldots,b-1\}$ and $u\in\{0,\ldots,b-1\}^*$. 
We have
$$
\begin{array}{rclr}
	S_b(nb + b-1)
	&=& S_b(\val_b(x00u(b-1))) & \\
	&=& 2 S_b(\val_b(x0u(b-1))) - S_b(\val_b(xu(b-1))) & \text{(by Lemma~\ref{lem:arbrex00})} \\
	&=& (2b-1) \, 2 S_b(\val_b(x0u)) - \sum_{s=0}^{b-2} \, 2 S_b(\val_b(x0us)) &  \\
	&&- (2b-1) S_b(\val_b(xu)) + \sum_{s=0}^{b-2} S_b(\val_b(xus)) &  \text{(by induction hypothesis)}\\
	&=& (2b-1) S_b(\val_b(x00u)) -  \sum_{s=0}^{b-2} S_b(\val_b(x00us)) & \text{(by Lemma~\ref{lem:arbrex00})} \\
	&=& (2b-1) S_b(n) -  \sum_{s=0}^{b-2} S_b(nb+s), & \text{(by Lemma~\ref{lem:arbrex00})}
\end{array}
$$
which proves~\eqref{eq:b-reg}.

The $\mathbb{Z}$-module $\left\langle \mathcal{K}_b(S_b)\right\rangle$ is thus generated by the $b$ sequences 
$$(S_b(n))_{n\ge 0}, (S_b(bn))_{n\ge 0}, (S_b(bn+1))_{n\ge 0}, \ldots, (S_b(bn+b-2))_{n\ge 0}.$$
\end{proof}

\begin{example}\label{ex:rel-base2}
Let $b=2$. Using Table~\ref{tab:ar}, Table~\ref{tab:cr0} and Table~\ref{tab:crs}, we find that $a_0=-1$, $a_1=1$, $a_2=4$, $a_3=5$, $c_{0,0}=2$, $c_{1,0}=1$, $c_{2,0}=-1$ and $c_{3,0}=-2$. 
In this case, there are no $c_{r,s}$ with $s>0$. 
Applying Theorem~\ref{thm:b-reg} and from~\eqref{eq:b-reg-2}, we get 
\begin{align*}
S_2(2n+1) &= 3S_2(n) - S_2(2n), \\
S_2(4n) &= -S_2(n) + 2 S_2(2n), \\
S_2(4n+1) &= S_2(n) + S_2(2n), \\
S_2(4n+2) &= 4S_2(n) - S_2(2n), \\
S_2(4n+3) &= 5S_2(n) -2 S_2(2n) 
\end{align*} 
for all $n\ge 0$. 
This result is a rewriting of {\cite[Theorem~21]{LRS2}}. 
Observe that the third and the fifth identities are redundant: they follow from the other ones. 
\end{example}

\begin{example}\label{ex:rel-base3}
Let $b=3$. Using Table~\ref{tab:ar}, Table~\ref{tab:cr0} and Table~\ref{tab:crs}, the values of the coefficients $a_r$, $c_{r,0}$ and $c_{r,1}$ can be found in Table~\ref{tab:val_coeff_base3}.
\begin{table}[h]
$$
\begin{array}{|c|c|c|c|c|c|c|c|c|c|}
\hline 
r & 0 & 1 & 2 & 3 & 4 & 5 & 6 & 7 & 8 \\
\hline
a_r & -1  & -2 & 3 & -2  & -1 & 3 & 8 & 8 & 9 \\
c_{r,0} & 2 & 2 & 1 & 1 & 0 & -1 & -1 & -2 & -2 \\
c_{r,1} & 0 & 1 & -1 & 2 & 2 & 1 & -2 & -1 & -2 \\
\hline
\end{array}
$$
\caption{The values of $a_r, c_{r,0}, c_{r,1}$ when $b=3$ and $r\in\{0,\ldots,8\}$.}
\label{tab:val_coeff_base3}
\end{table}
Applying Theorem~\ref{thm:b-reg} and from~\eqref{eq:b-reg-2}, we get 
\begin{align*}
S_3(3n+2) &= 5S_3(n) - S_3(3n) - S_3(3n+1), \\
S_3(9n) &= -S_3(n) + 2 S_3(3n), \\
S_3(9n+1) &= -2S_3(n) + 2 S_3(3n) + S_3(3n+1), \\
S_3(9n+2) &= 3S_3(n) + S_3(3n) - S_3(3n+1), \\
S_3(9n+3) &= -2S_3(n) + S_3(3n) + 2S_3(3n+1), \\
S_3(9n+4) &= -S_3(n) + 2 S_3(3n+1), \\
S_3(9n+5) &= 3S_3(n) - S_3(3n) +S_3(3n+1), \\
S_3(9n+6) &= 8S_3(n) - S_3(3n) -2S_3(3n+1), \\
S_3(9n+7) &= 8S_3(n) - 2S_3(3n) -S_3(3n+1), \\
S_3(9n+8) &= 9S_3(n) - 2S_3(3n) -2 S_3(3n+1)
\end{align*} 
for all $n\ge 0$. 
This result is a proof of {\cite[Conjecture~26]{LRS2}}.
Observe that the fourth, the seventh and the tenth identities are redundant.
\end{example}

\begin{remark}\label{rem:nombre-relations}
Combining~\eqref{eq:b-reg} and~\eqref{eq:b-reg-2} yield $b^2+1$ identities to generate the $\mathbb{Z}$-module $\left\langle \mathcal{K}_b(S_b)\right\rangle$. However, as illustrated in Example~\ref{ex:rel-base2} and Example~\ref{ex:rel-base3}, only $b^2-b+1$ identities are useful: the relations established for the sequences $(S_b(b^2n+br+b-1))_{n\ge 0}$, with $r\in\{0,\ldots,b-1\}$, can be deduced from the other identities. 
\end{remark}

\begin{remark}\label{rem:rep-lin}
Using Theorem~\ref{thm:b-reg} and~\eqref{eq:b-reg-2} and the set of $b$ generators of the $\mathbb{Z}$-module $\left\langle \mathcal{K}_b(S_b)\right\rangle$ being
$$
\{ (S_b(n))_{n\ge 0}, (S_b(bn))_{n\ge 0}, (S_b(bn+1))_{n\ge 0}, \ldots, (S_b(bn+b-2))_{n\ge 0} \},
$$ 
we get matrices to compute $S_b(n)$ in a number of steps proportional to $\log_b(n)$. 
For all $n\ge 0$, let 
$$
V_b(n) = \left(
\begin{array}{c}
 S_b(n) \\
 S_b(bn)\\
 S_b(bn+1)\\
 \vdots\\
 S_b(bn+b-2) 
\end{array}
\right) \in \mathbb{Z}^{b} .
$$
Consider the matrix-valued morphism $\mu_b : \{0,1,\ldots,b-1\}^* \to \mathbb{Z}^{b}_{b}$ defined, for all $s\in\{0,\ldots,b-2\}$, by
$$
\mu_b(s)=
\left(\begin{array}{cccccccc}
 0 & 0 & \cdots & 0 & 1 & 0 & \cdots & 0 
 \\ 
 a_{bs} & c_{bs,0} & \cdots & c_{bs,s-1} & c_{bs,s} & c_{bs,s+1} & \cdots & c_{bs,b-2}
 \\
a_{bs+1} & c_{bs+1,0} & \cdots & c_{bs+1,s-1} & c_{bs+1,s} & c_{bs+1,s+1} & \cdots & c_{bs+1,b-2}
\\ 
\vdots & \vdots &  & \vdots & \vdots & \vdots &  & \vdots 
\\
a_{bs+b-2} & c_{bs+b-2,0} & \cdots & c_{bs+b-2,s-1} & c_{bs+b-2,s} & c_{bs+b-2,s+1} & \cdots & c_{bs+b-2,b-2} \\
\end{array}\right)
$$ 
and 
$$
\mu_b(b-1)=
\left(\begin{array}{ccccc}
 (2b-1) & -1 & -1 & \cdots & -1 \\ 
 a_{b(b-1)} & c_{b(b-1),0} & c_{b(b-1),1} & \cdots & c_{b(b-1),b-2} \\
  a_{b(b-1)+1} & c_{b(b-1)+1,0} & c_{b(b-1)+1,1} & \cdots & c_{b(b-1)+1,b-2} \\
  \vdots & \vdots & \vdots & & \vdots \\
    a_{b(b-1)+b-2} & c_{b(b-1)+b-2,0} & c_{b(b-1)+b-2,1} & \cdots & c_{b(b-1)+b-2,b-2} \\
\end{array}\right).
$$ 
Observe that the number of generators explains the size of the matrices above. 
For each $s\in\{0,\ldots,b-2\}$, exactly $b-1$ identities from Theorem~\ref{thm:b-reg} are used to define the matrix $\mu_b(s)$. If $s,s'\in\{0,\ldots,b-2\}$ are such that $s\ne s'$, then the relations used to define the matrices $\mu_b(s)$ and $\mu_b(s')$ are pairwise distinct. Finally, the first row of the matrix $\mu_b(b-1)$ is~\eqref{eq:b-reg-2} and the other rows are $b-1$ identities from Theorem~\ref{thm:b-reg}, which are distinct from the previous relations. Consequently, $(b-1)(b-1)+b$ identities are used, which corroborates Remark~\ref{rem:nombre-relations}.

Using the definition of the morphism $\mu$, we can show that $V_b(bn+s)=\mu_b(s)V_b(n)$ for all $s\in\{0,\ldots,b-1\}$ and $n\ge 0$. Consequently, if $\rep_b(n)= n_k \cdots n_0$, then 
$$
S_b(n) =
\begin{pmatrix}
    1&0&\cdots&0\\
\end{pmatrix}
\, \mu_b(n_0)\cdots \mu_b(n_k)\, V_b(0).
$$

For example, when $b=2$, the matrices $\mu_2(0)$ and $\mu_2(1)$ are those given in {\cite[Corollary~22]{LRS2}}.
When $b=3$, we get 
$$\mu_3(0)=\left(\begin{array}{ccc}
 0 & 1 & 0 \\ 
 -1 & 2 & 0 \\
 -2 & 2 & 1\\
\end{array}\right),\quad
\mu_3(1)=\left(
\begin{array}{ccc}
 0 & 0 & 1 \\
 -2 & 1 & 2 \\
 -1 & 0 & 2 \\
\end{array}
\right), \quad
\mu_3(2)=\left(
\begin{array}{ccc}
5 & -1 & -1 \\
8 & -1 & -2\\
8 & -2 & -1\\
\end{array}
\right).$$
\end{remark}

The class of $b$-synchronized sequences is intermediate between the classes of $b$-automatic sequences and $b$-regular sequences. 
These sequences were first introduced in \cite{CM}.

\begin{prop}\label{pro:pasb-syn}
The sequence $(S_b(n))_{n\ge 0}$ is not $b$-synchronized.
\end{prop}
\begin{proof}
The proof is exactly the same as {\cite[Proposition~24]{LRS2}}.
\end{proof}

To conclude this section, the following result proves that the sequence $(S_b(n))_{n\ge 0}$ has a partial palindromic structure as the sequence $(S_2(n))_{n\ge 0}$; see \cite{LRS2}. For instance, the sequence $(S_3(n))_{n\ge 0}$ is depicted in Figure~\ref{fig:palindromie3} inside the interval $[2\cdot 3^4,3^5]$. 
\begin{figure}[h!tb]
    \centering
\includegraphics[scale=0.5]{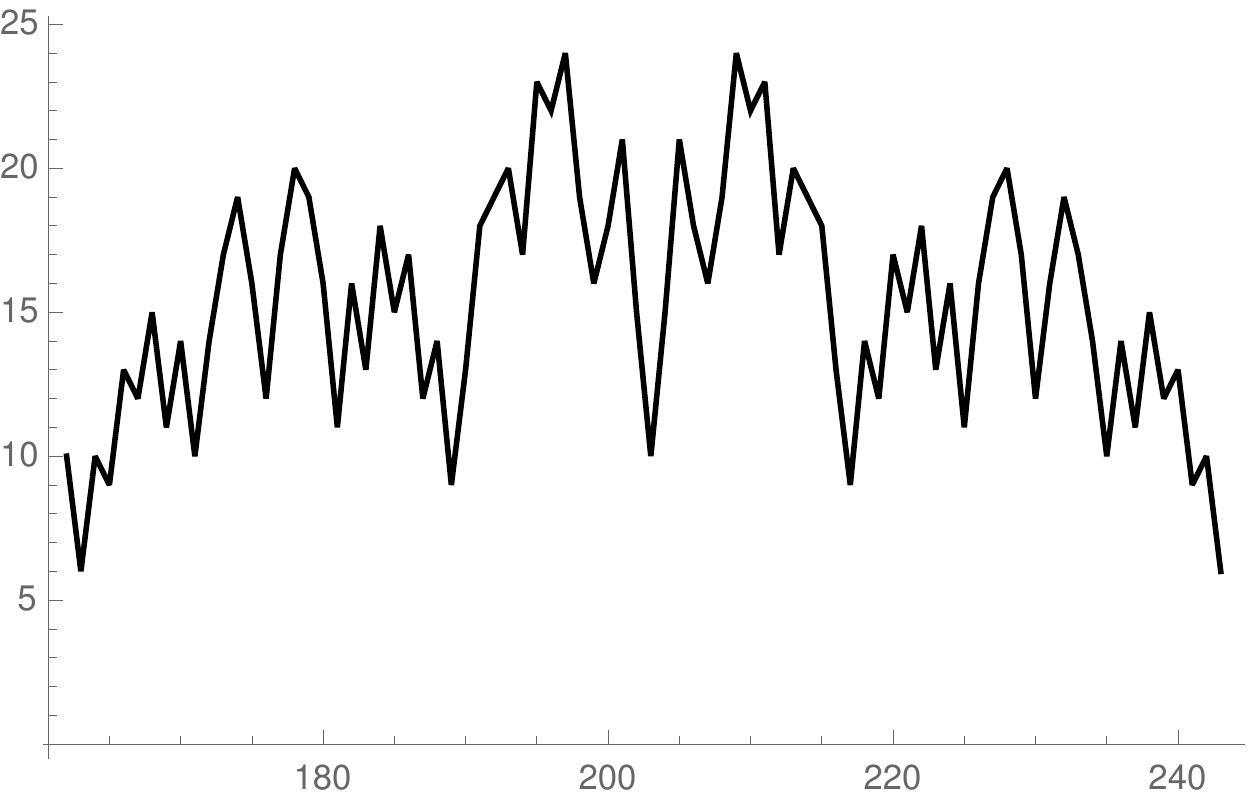}
    \caption{The sequence $(S_3(n))_{n\ge 0}$ inside the interval $[2\cdot 3^4,3^5]$.}
    \label{fig:palindromie3}
\end{figure}

\begin{prop}\label{pro:bpalindrome} 
Let $u$ be a word in $\{0,1,\ldots,b-1\}^*$. 
Define $\bar{u}$ by replacing in $u$ every letter $a\in \{0,1,\ldots,b-1\}$ by the letter $(b-1)-a\in \{0,1,\ldots,b-1\}$. 
Then
$$
\#\left\{ v\in L_b \mid \binom{(b-1)u}{v} > 0 \right\} = \#\left\{ v\in L_b \mid \binom{(b-1)\bar{u}}{v} > 0 \right\}.
$$
In particular, there exists a palindromic substructure inside of the sequence $(S_b(n))_{n\ge 0}$, i.e., for all $\ell\ge 1$ and $0\le r < b^\ell$,
$$
S_b((b-1)\cdot b^{\ell}+r) = S_b((b-1)\cdot b^{\ell}+b^{\ell}-r-1).  
$$
\end{prop}
\begin{proof}
The trees $\mathcal{T}((b-1)u)$ and $\mathcal{T}((b-1)\bar{u})$ are isomorphic. 
Indeed, on the one hand, each node of the form $(b-1)x$ in the first tree corresponds to the node $(b-1)\bar{x}$ in the second one and conversely. 
On the other hand, if there exist letters $a\in\{1,\ldots,b-2\}$ in the word $(b-1)u$, the position of the first letter $a$ in the word $(b-1)u$ is equal to the position of the first letter $(b-1)-a$ in the word $(b-1)\bar{u}$ and conversely. Consequently, the node of the form $ax$ in the first tree corresponds to the node of the form $((b-1)-a)\bar{x}$ in the second tree and conversely. 

For the special case, note that for every word $z$ of length $\ell$, there exists $r\in\{0,\ldots,b^\ell-1\}$ such that $\rep_b((b-1)\cdot b^\ell +r)=(b-1)z$ and
$$\val_b(\bar{z})=b^\ell-1-r \in\{0\ldots,b^\ell-1\}.$$
Hence, $(b-1)\bar{z}=\rep_b((b-1)\cdot b^\ell+b^\ell-1-r)$. 
Using~\eqref{def:S_b}, we obtain the desired result.
\end{proof}

\section{Asymptotics of the summatory function $(A_b(n))_{n\ge 0}$}\label{sec:summatory}

In this section, we consider the summatory function $(A_b(n))_{n\ge 0}$ of the sequence $(S_b(n))_{n\ge 0}$; see Definition~\ref{def:S_b}. 
The aim of this section is to apply the method introduced in \cite{LRS3} to obtain the asymptotic behavior of $(A_b(n))_{n\ge 0}$. 
As an easy consequence of the $b$-regularity of $(S_b(n))_{n\ge 0}$, we have the following result. 

\begin{prop}
For all $b\ge 2$, the sequence $(A_b(n))_{n\ge 0}$  is $b$-regular.
\end{prop}
\begin{proof}
This is a direct consequence of Theorem~\ref{thm:b-reg} and of the fact that the summatory function of a $b$-regular sequence is also $b$-regular; see \cite[Theorem~16.4.1]{AS03}.
\end{proof}

From a linear representation with matrices of size $d\times d$ associated with a $b$-regular sequence, one can derive a linear representation with matrices of size $2d \times 2d$ associated with its summatory function; see \cite[Lemma 1]{Dumas13}. 
Consequently, using Remark~\ref{rem:rep-lin}, one can obtain a linear representation with matrices of size $2b\times 2b$ for the summatory function $(A_b(n))_{n\ge 0}$. 
The goal is to decompose $(A_b(n))_{n\ge 0}$ into linear combinations of powers of $(2b-1)$. 
 We need the following two lemmas.

\begin{lemma}\label{lem:sommatoire_aux_puissances}
For all $\ell\ge 0$ and all $x\in\{1,\ldots,b-1\}$, we have
$$
A_b(xb^\ell) = (2x-1) \cdot (2b-1)^\ell.
$$
\end{lemma}
\begin{proof}
We proceed by induction on $\ell\ge 0$. If $\ell=0$ and $x\in\{1,\ldots,b-1\}$, then using Table~\ref{tab:initialSb}, we have
$$
A_b(x) 
= S_b(0) + \sum_{j=1}^{x-1} S_b(j) 
= 2x-1.
$$ 
If $\ell=1$ and $x\in\{1,\ldots,b-1\}$, then we have
$$
A_b(xb) = A_b(b) + \sum_{y=1}^{x-1} \sum_{j=0}^{b-1} S_b(yb+j).
$$ 
Using Table~\ref{tab:initialSb}, we get $A_b(xb)= (2x-1)(2b-1)$.

Now suppose that $\ell\ge 1$ and assume that the result holds for all $\ell' \le \ell$. 
To prove the result, we again proceed by induction on $x\in\{1,\ldots,b-1\}$. 
When $x=1$, we must show that $A_b(b^{\ell+1})=(2b-1)^{\ell+1}$. 
We have
\[
	A_b(b^{\ell+1})
	= 
	A_b(b^\ell) + \sum_{y=1}^{b-1}\sum_{j=0}^{b^\ell-1} S_b(yb^\ell+j).
\]
By decomposing the sum into three parts accordingly to Proposition~\ref{prop:rel-s}, we get
\begin{eqnarray*}
	A_b(b^{\ell+1})
	&=& 
	A_b(b^\ell) + 	\sum_{y=1}^{b-1} \sum_{j=0}^{b^{\ell-1}-1} S_b(yb^\ell+j)
	+ \sum_{y=1}^{b-1} \sum_{j=0}^{b^{\ell-1}-1} S_b(yb^\ell +yb^{\ell-1}+j)\\
	&+&
	\sum_{y=1}^{b-1} \sum_{\substack{1 \le z \le b-1 \\ z\ne y}} \sum_{j=0}^{b^{\ell-1}-1} S_b(yb^\ell + zb^{\ell-1} +j), 
	\end{eqnarray*}
and, using Proposition~\ref{prop:rel-s},  
\begin{eqnarray}
	A_b(b^{\ell+1})
	&=& A_b(b^{\ell}) \nonumber \\
	&+&
	\sum_{y=1}^{b-1} \sum_{j=0}^{b^{\ell-1}-1} (S_b(yb^{\ell-1}+j)+S_b(j)) \label{eq:terme2} \\
	&+& 
	\sum_{y=1}^{b-1} \sum_{j=0}^{b^{\ell-1}-1} (2S_b(yb^{\ell-1}+j)-S_b(j)) \label{eq:terme3} \\
	&+ &
	\sum_{y=1}^{b-1} \sum_{\substack{1 \le z \le b-1 \\ z\ne y}} \sum_{j=0}^{b^{\ell-1}-1} (S_b(yb^{\ell-1}+j)+2S_b(zb^{\ell-1}+j)-2S_b(j)) \label{eq:terme4}.
\end{eqnarray}
By observing that for all $y$, 
\begin{equation}
\label{eq:calcul1}
	\sum_{j=0}^{b^{\ell-1}-1} S_b(yb^{\ell-1}+j) = A_b((y+1)b^{\ell-1})-A_b(yb^{\ell-1})
\qquad
\text{and}
\qquad	
\sum_{j=0}^{b^{\ell-1}-1} S_b(j)) = A_b(b^{\ell-1}),
\end{equation}
and that
\begin{equation}
\label{eq:calcul2}
	\sum_{y=1}^{b-1} \left(A_b((y+1)b^{\ell-1})-A_b(yb^{\ell-1})\right) = A_b(b^\ell) - A_b(b^{\ell-1}),
\end{equation}
we obtain
\begin{eqnarray*}
\eqref{eq:terme2} &=& A_b(b^{\ell}) + (b-2) A_b(b^{\ell-1}), \\
\eqref{eq:terme3} &=&
2 A_b(b^{\ell}) - (b+1) A_b(b^{\ell-1}),\\
\eqref{eq:terme4} &=& 3(b-2)(A_b(b^{\ell})-A_b(b^{\ell-1})) - 2(b-1)(b-2)A_b(b^{\ell-1})
=
3(b-2)A_b(b^{\ell}) - (b-2)(2b+1)A_b(b^{\ell-1}),
\end{eqnarray*}
and finally
\[
A_b(b^{\ell+1})
=
(3b-2) A_b(b^{\ell})
-
(2 b^2 - 3b + 1) A_b(b^{\ell-1}).
\]
Using the induction hypothesis, we obtain
\[
A_b(b^{\ell+1})
=
(3b-2) (2b-1)^{\ell}
-
(2 b^2 - 3b + 1) (2b-1)^{\ell-1} 
=
(2b-1)^{\ell+1},
\]
which ends the case where $x=1$. 

Now suppose that $x\in\{2,\ldots,b-1\}$ and assume that the result holds for all $x' < x$. 
The proof follows the same lines as in the case $x=1$ with the difference that we decompose the sum into 
\begin{eqnarray*}
	A_b(xb^{\ell+1})
	&=& 
	A_b((x-1)b^{\ell+1}) + \sum_{j=0}^{b^{\ell+1}-1} S_b((x-1)b^{\ell+1}+j)\\
	&=&
	A_b((x-1)b^{\ell+1})
	+
	\sum_{j=0}^{b^{\ell}-1} S_b((x-1)b^{\ell+1} +j)
	+
	\sum_{j=0}^{b^{\ell}-1} S_b((x-1)b^{\ell+1} +(x-1)b^{\ell}+j) \\
	&+&
	\sum_{\substack{1 \le y \le b-1 \\ y\ne x-1}} \sum_{j=0}^{b^{\ell}-1} S_b((x-1)b^{\ell+1} + yb^{\ell} +j).
\end{eqnarray*}
Applying Proposition~\ref{prop:rel-s} and using~\eqref{eq:calcul1} and~\eqref{eq:calcul2} leads to the equality
\[
A_b(xb^{\ell+1}) 
= 
A_b((x-1)b^{\ell+1})
+ (b-1) A_b(x b^{\ell})
- (b-1) A_b((x-1)b^{\ell})
+ 2 A_b(b^{\ell+1})
- 2 (b-1) A_b(b^{\ell}).
\]
The induction hypothesis ends the computation.
\end{proof}

\begin{lemma}\label{lem:sommatoire_aux_val_part}
For all $\ell\ge 1$ and all $x,y\in\{1,\ldots,b-1\}$, we have
$$
A_b(xb^\ell+yb^{\ell-1})= 
\begin{cases}
(4xb-2x+4y-2b) \cdot (2b-1)^{\ell-1}, & \text{if } y \le x; \\
(4xb-2x+4y-2b-1) \cdot (2b-1)^{\ell-1}, & \text{if } y > x.
\end{cases} 
$$
\end{lemma}
\begin{proof}
The proof of this lemma is similar to the proof of Lemma~\ref{lem:sommatoire_aux_puissances} so we only proof the formula for $A_b(xb^\ell+xb^{\ell-1})$, the other being similarly handled. 
We proceed by induction on $\ell\ge 1$. If $\ell=1$, the result follows from Table~\ref{tab:initialSb}.  
Assume that $\ell\ge 2$ and that the formulas hold for all $\ell' <\ell$. 
We have
\begin{align*}
	A_b(xb^{\ell}+xb^{\ell-1})
	&= A_b(xb^\ell) + \sum_{j=0}^{b^{\ell-1}-1} S_b(xb^\ell+j) + \sum_{y=1}^{x-1}\sum_{j=0}^{b^{\ell-1}-1} S_b(xb^\ell+yb^{\ell-1}+j).
\end{align*}
Applying Proposition~\ref{prop:rel-s} and using~\eqref{eq:calcul1} and~\eqref{eq:calcul2} leads to the equality
\[
A_b(xb^{\ell}+xb^{\ell-1})
=
A_b(xb^\ell) + x A_b((x+1)b^{\ell-1}) + (2-x) A_b(xb^{\ell-1}) + (1-2x) A_b(b^{\ell-1}).
\]
Using Lemma~\ref{lem:sommatoire_aux_puissances} completes the computation.
\end{proof}

Lemma~\ref{lem:sommatoire_aux_puissances} and Lemma~\ref{lem:sommatoire_aux_val_part} give rise to recurrence relations satisfied by the summatory function $(A_b(n))_{n\ge 0}$ as stated below. This is a key result that permits us to introduce $(2b-1)$-decompositions (Definition~\ref{def:2b-1decomp} below) of the summatory function $(A_b(n))_{n\ge 0}$ and allows us to easily deduce Theorem~\ref{thm:asymptob}; see~\cite{LRS3} for similar results in base $2$. 

\begin{prop}\label{prop:rel-A}
For all $x,y\in\{1,\ldots,b-1\}$ with $x\ne y$, all $\ell\ge 1$ and all $r\in\{0,\ldots,b^{\ell-1}\}$,
\begin{eqnarray}
\label{eq:rec_Ab_1}
A_b(xb^\ell+r) 
&=& 
(2b-2)\cdot (2x-1)\cdot (2b-1)^{\ell-1} + A_b(xb^{\ell-1}+r) + A_b(r); 
\\
\label{eq:rec_Ab_2}
A_b(xb^\ell+xb^{\ell-1}+r) 
&=& 
(4xb-2x-2b+2) \cdot (2b-1)^{\ell-1} +  2 A_b(xb^{\ell-1}+r) - A_b(r);
\\
\label{eq:rec_Ab_3}
A_b(xb^\ell+yb^{\ell-1}+r) 
&=&
\begin{cases}
(4xb-4x-2b+3) \cdot (2b-1)^{\ell-1} + A_b(xb^{\ell-1}+r) & \\
+ 2 A_b(yb^{\ell-1}+r) - 2 A_b(r), & \text{if } y < x; \\
(4xb-4x-2b+2) \cdot (2b-1)^{\ell-1} + A_b(xb^{\ell-1}+r) & \\
+ 2 A_b(yb^{\ell-1}+r) - 2 A_b(r), & \text{if } y > x.
\end{cases}
\end{eqnarray}
\end{prop}
\begin{proof}
We first prove~\eqref{eq:rec_Ab_1}. Let $x\in\{1,\ldots,b-1\}$, $\ell\ge 1$ and $r\in\{0,\ldots,b^{\ell-1}\}$. 
If $r=0$, then~\eqref{eq:rec_Ab_1} holds using Lemma~\ref{lem:sommatoire_aux_puissances}. 
Now suppose that $r\in\{1,\ldots,b^{\ell-1}\}$. 
Applying successively Proposition~\ref{prop:rel-s} and Lemma~\ref{lem:sommatoire_aux_puissances}, we have
\begin{align*}
	A_b(xb^{\ell}+r)
	&= A_b(xb^\ell) + \sum_{j=0}^{r-1} S_b(xb^{\ell}+j) \\
	&= A_b(xb^\ell) + \sum_{j=0}^{r-1} (S_b(xb^{\ell-1}+j)+S_b(j))\\
	&= A_b(xb^\ell) + (A_b(xb^{\ell-1}+r)-A_b(xb^{\ell-1}))+A_b(r)\\	
	&= (2b-2)(2x-1)(2b-1)^{\ell-1} + A_b(xb^{\ell-1}+r)+A_b(r),
\end{align*}
which proves~\eqref{eq:rec_Ab_1}.

The proof of~\eqref{eq:rec_Ab_2} and~\eqref{eq:rec_Ab_3} are similar, thus we only prove~\eqref{eq:rec_Ab_2}. 
Let $x\in\{1,\ldots,b-1\}$, $\ell\ge 1$ and $r\in\{0,\ldots,b^{\ell-1}\}$. 
If $r=0$, then~\eqref{eq:rec_Ab_2} holds using Lemma~\ref{lem:sommatoire_aux_val_part}. 
Now suppose that $r\in\{1,\ldots,b^{\ell-1}\}$. 
Applying Proposition~\ref{prop:rel-s}, we have
\begin{align*}
	A_b(xb^{\ell}+xb^{\ell-1}+r)
	&= A_b(xb^{\ell}+xb^{\ell-1}) + \sum_{j=0}^{r-1} S_b(xb^{\ell}+xb^{\ell-1}+j) \\
	&= A_b(xb^{\ell}+xb^{\ell-1}) + \sum_{j=0}^{r-1} (2S_b(xb^{\ell-1}+j)-S_b(j))\\
	&= A_b(xb^{\ell}+xb^{\ell-1}) + 2 (A_b(xb^{\ell-1}+r)-A_b(xb^{\ell-1}))-A_b(r).
\end{align*}
Using Lemma~\ref{lem:sommatoire_aux_puissances} and Lemma~\ref{lem:sommatoire_aux_val_part}, we get
\begin{align*}
A_b(xb^{\ell}+xb^{\ell-1}+r)
	&= (4xb+2x-2b) (2b-1)^{\ell-1} - 2(2x-1) (2b-1)^{\ell-1} + 2 A_b(xb^{\ell-1}+r)-A_b(r)\\
	&= (4xb-2x-2b+2) (2b-1)^{\ell-1} + 2 A_b(xb^{\ell-1}+r)-A_b(r),
\end{align*}
which proves~\eqref{eq:rec_Ab_2}.
\end{proof}

The following corollary was conjectured in \cite{LRS3}.

\begin{corollary}\label{cor:mult}
For all $n\ge 0$, we have $A_b(nb)=(2b-1)A_b(n)$.
\end{corollary}
\begin{proof}
Let us proceed by induction on $n\ge 0$. 
It is easy to check by hand that the result holds  for $n\in\{0,\ldots,b-1\}$. 
Thus consider $n \ge b$ and suppose that the result holds for all $n' < n$. 
The reasoning is divided into three cases according to the form of the base-$b$ expansion of $n$.  
As a first case, we write $n=xb^\ell+r$ with $x\in\{1,\ldots,b-1\}$, $\ell\ge 1$ and $0\le r < b^{\ell-1}$. 
By Proposition~\ref{prop:rel-A}, we have
\begin{align*}
A_b(nb)-(2b-1)A_b(n) &= (2b-2)\cdot (2x-1)\cdot (2b-1)^{\ell} + A_b(xb^{\ell}+br) + A_b(br) - (2b-2)\cdot (2x-1)\cdot (2b-1)^{\ell} \\
& - (2b-1) A_b(xb^{\ell-1}+r) -(2b-1) A_b(r)  
\end{align*}
We conclude this case by using the induction hypothesis. 
The other cases can be handled using the same technique.
\end{proof}

Using Proposition~\ref{prop:rel-A}, we can define $(2b-1)$-decompositions as follows. 

\begin{definition}\label{def:2b-1decomp}
Let $n\ge b$. 
Applying iteratively Proposition~\ref{prop:rel-A} provides a unique decomposition of the form 
$$A_b(n)=\sum_{i=0}^{\ell_b(n)} d_i(n)\, (2b-1)^{\ell_b(n)-i}$$
where $d_i(n)$ are integers, $d_{0}(n)\neq 0$ and $\ell_b(n)$ stands for $\lfloor \log_b n \rfloor - 1$. 
We say that the word 
$$
d_0(n) \cdots d_{\ell_b(n)}(n)
$$ 
is the {\em $(2b-1)$-decomposition} of $A_b(n)$. 
For the sake of clarity, we also write $(d_0(n),\ldots,d_{\ell_b(n)}(n))$. 
Also notice that the notion of $(2b-1)$-decomposition is only valid for integers in the sequence $(A_b(n))_{n\ge 0}$.
\end{definition}

\begin{example}
Let $b=3$. Let us compute the $5$-decomposition of $A_3(150)=1665$. We have $\rep_3(150)=12120$ and $\ell_3(150)=3$. Applying once Proposition~\ref{prop:rel-A} leads to
\begin{equation}\label{eq:exA3(150)}
A_3(150)=A_3(3^4+2\cdot 3^3 + 15) = 4\cdot 5^3 + A_3(3^3+15) + 2 A_3(2\cdot 3^3+15) - 2 A_3(15).
\end{equation}
Applying again Proposition~\ref{prop:rel-A}, we get
\begin{align*}
A_3(3^3+15) &= A_3(3^3+3^2+6) = 6\cdot 3^2 + 2 A_3(3^2+6) - A_3(6), \\
A_3(2\cdot 3^3+15) &= A_3(2\cdot 3^3+3^2+6) = 13\cdot 3^2 + A_3(2\cdot 3^2+6) + 2A_3(3^2+6) - 2A_3(6), \\
A_3(15)&=A_3(3^2+2\cdot 3^1) = 4\cdot 5^1 + A_3(3^1) + 2A_3(2\cdot 3^1) - 2A_3(0).
\end{align*}
Using Proposition~\ref{prop:rel-A}, we find
\begin{align*}
A_3(3^2+6) &= A_3(3^2+2\cdot 3^1) = 4\cdot 5^1 + A_3(3^1) + 2 A_3(2\cdot 3^1) - 2A_3(0), \\
A_3(2\cdot 3^2+6)&= A_3(2\cdot 3^2+2\cdot 3^1) = 16\cdot 5^1 + 2A_3(2\cdot 3^1)-A_3(0), \\
A_3(6) &= A_3(2\cdot 3^1) = 12\cdot 5^0 + A_3(2\cdot 3^0) + A_3(0) = 15\cdot 5^0.
\end{align*}
Using Lemma~\ref{lem:sommatoire_aux_puissances}, we have $A_3(3^1)=5^1$ and $A_3(2\cdot 3^1)= 3\cdot 5^1$. Plugging all those values together in~\eqref{eq:exA3(150)}, we finally have
$$
A_3(150) = 4\cdot 5^3 + 32\cdot 5^2 + 82\cdot 5^1 - 45 \cdot 5^0. 
$$
The $5$-decomposition of $A_3(150)$ is thus $(4,32,82,-45)$.

\end{example}

The proof of the next result follows the same lines as the proof of~\cite[Theorem~1]{LRS3}. 
Therefore we only sketch it. 

\begin{theorem}\label{thm:asymptob}
There exists a continuous and periodic function $\mathcal{H}_b$ of period 1 such that, for all large enough $n$, 
$$A_b(n) =(2b-1)^{\log_b n}\ \mathcal{H}_b(\log_b n).$$ 
\end{theorem}
As an example, when $b=3$, the function $\mathcal{H}_3$ is depicted in Figure~\ref{fig:H3} over one period. 
\begin{figure}[h!tb]
    \centering
\includegraphics[scale=0.3]{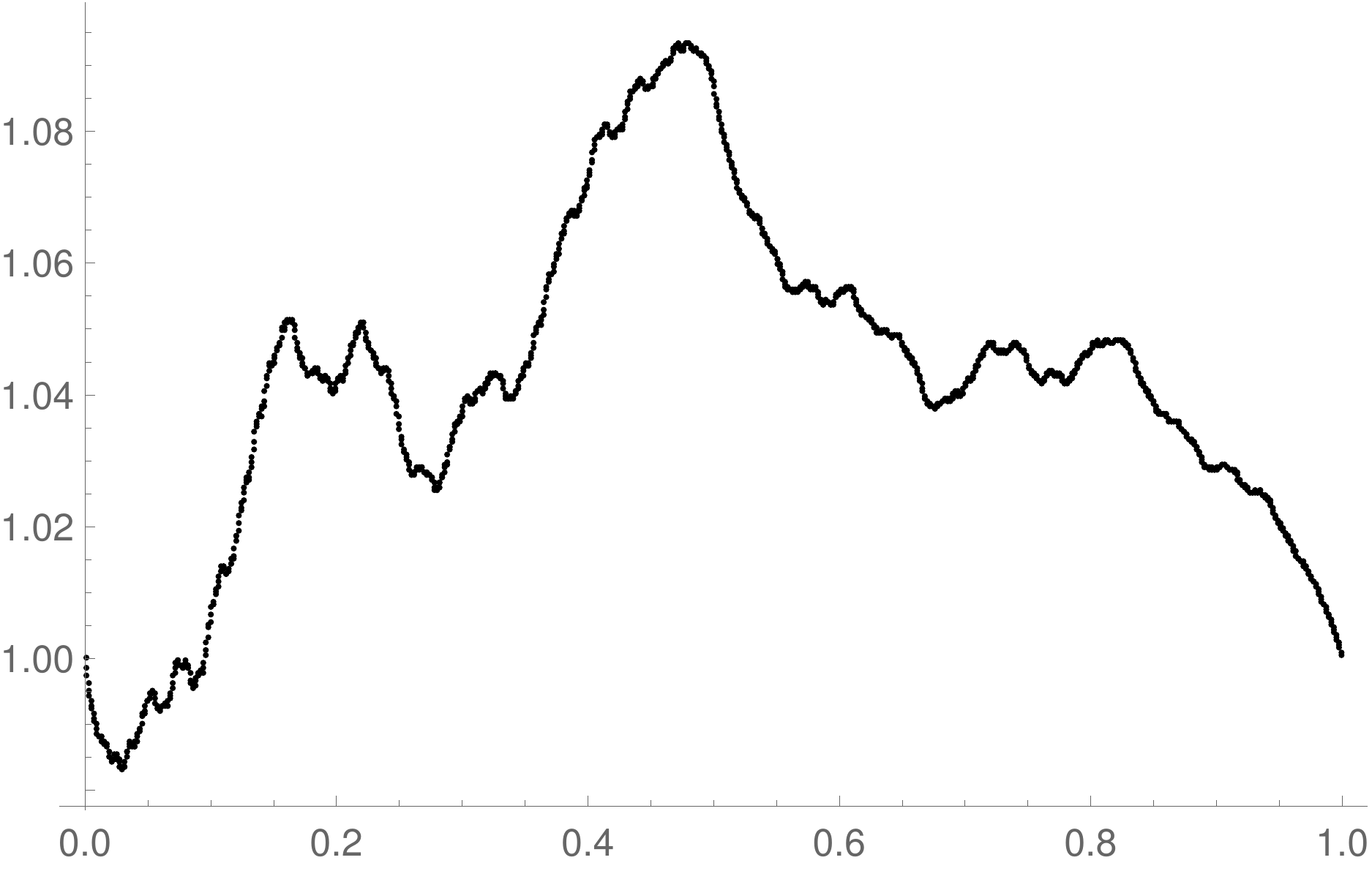}
    \caption{The function $\mathcal{H}_3$ over one period.}
    \label{fig:H3}
\end{figure}

\begin{proof}[Sketch of the proof of Theorem~\ref{thm:asymptob}]
Let us start by defining the function $\mathcal{H}_b$. Given any integer $n\ge 1$, we let $\phi_n$ denote the function 
$$
\phi_n(\alpha) = \frac{A_b(e_n(\alpha))}{(2b-1)^{\log_b(e_n(\alpha))}}, \quad \alpha\in[0,1)
$$
where $e_n(\alpha)=b^{n+1} + b \lfloor b^n \alpha \rfloor + 1$. With a proof analogous to the one of~\cite[Proposition~20]{LRS3}, the sequence of functions $(\phi_n)_{n\ge 1}$ uniformly converges to a function $\Phi_b$. As in~\cite[Theorem~5]{LRS3}, this function is continuous on $[0,1]$ and such that $\Phi_b(0)=\Phi_b(1)=1$. Furthermore, it satisfies
$$
A_b(b^k+r) = (2b-1)^{\log_b(b^k+r)} \Phi_b\left(\frac{r}{b^k}\right) \quad k\ge 1, 0\le r< b^k;
$$
see~\cite[Lemma~24]{LRS3}. Using Corollary~\ref{cor:mult}, we get that, for all $n=b^j(b^k+r)$, $j,k\ge 0$ and $r\in\{0,\ldots,b^k-1\}$, 
$$
A_b(n) = (2b-1)^{\log_b(n)}\Phi_b\left(\frac{r}{b^k}\right).
$$
The function $\mathcal{H}_b$ is defined by $\mathcal{H}_b(x)=\Phi_b( b^{\{x\}}-1 )$ for all real $x$ ($\{\cdot \}$ stands for the fractional part).
\end{proof}

\bigskip
\hrule
\bigskip

\noindent 2010 {\it Mathematics Subject Classification}: 11A63, 11B65, 11B85, 41A60, 68R15.

\noindent \emph{Keywords:}
Binomial coefficients, subwords, generalized Pascal triangles, base-$b$ expansions, regular sequences, summatory function, asymptotic behavior

\bigskip
\hrule
\bigskip

\noindent (Concerned with sequences
\seqnum{A007306},
\seqnum{A282714},
\seqnum{A282715},
\seqnum{A282720},
\seqnum{A282728},
\seqnum{A284441},
and \seqnum{A284442}.)

\bigskip
\hrule
\bigskip

\end{document}